\documentclass[11pt]{amsart}
\usepackage{amssymb}
\usepackage{a4wide}
\usepackage{paralist}
\usepackage{color}
\usepackage{graphicx}
\usepackage{bbm}

\newtheorem{theorem}{Theorem}[section]
\newtheorem*{theoremstar}{Theorem~\ref{thm:binomialzeros}}
\newtheorem{lemma}[theorem]{Lemma}
\newtheorem{proposition}[theorem]{Proposition}
\newtheorem{corollary}[theorem]{Corollary}
\newtheorem{observation}[theorem]{Observation}

\theoremstyle{definition}
\newtheorem{remark}[theorem]{Remark}
\newtheorem{example}[theorem]{Example}
\newtheorem{definition}[theorem]{Definition}

\newcommand{\CC}{\mathbbm{C}}
\newcommand{\NN}{\mathbbm{N}}
\newcommand{\RR}{\mathbbm{R}}
\newcommand{\e}{\varepsilon}
\newcommand{\oW}{\overline{W}}

\newcommand{\tW}{\widetilde{W}}
\newcommand{\ta}{\tilde a}
\newcommand{\ow}{\overline{w}}

\newcommand{\bz}{{\bar z}}

\newcommand{\bJ}{{\bar J}}

\newcommand{\cS}{\mathcal{S}}
\newcommand{\DD}{\mathcal{D}}
\newcommand{\cB}{\mathcal{B}}

\DeclareMathOperator{\Real}{Re}
\DeclareMathOperator{\Imag}{Im}

\DeclareMathOperator{\sign}{sign}
\DeclareMathOperator{\cl}{cl}

\newcommand{\fl}[1]{\left\lfloor{#1}\right\rfloor}
\newcommand{\iverson}[1]{\left[\!\left[{#1}\right]\!\right]}
\newcommand{\pos}[1]{{}_{\RR_{\ge0}}\langle{#1}\rangle}

\graphicspath{{graphics/}}

 \setdefaultleftmargin{0pt}{0pt}{}{}{}{} 
\begin{document}

\title{Gale duality bounds for roots \linebreak of
  polynomials with nonnegative coefficients}
\author{Julian Pfeifle}
\address{Departament de Matem\`atica Aplicada II, Universitat
  Polit\`ecnica de Catalunya}
\email{julian.pfeifle@upc.edu}
\begin{abstract}
  We bound the location of roots of polynomials that have nonnegative
  coefficients with respect to a fixed but arbitrary basis of the
  vector space of polynomials of degree at most~$d$. For this, we
  interpret the basis polynomials as vector fields in the real plane,
  and at each point in the plane analyze the combinatorics of the Gale
  dual vector configuration. This approach permits us to incorporate
  arbitrary linear equations and inequalities among the coefficients
  in a unified manner to obtain more precise bounds on the location of
  roots. We apply our technique to bound the location of roots of
  Ehrhart and chromatic polynomials.  Finally, we give an explanation
  for the clustering seen in plots of roots of random polynomials.
\end{abstract}
\maketitle

\section{Introduction}

The Ehrhart polynomial of a $d$-dimensional lattice polytope $Q$ is a
real polynomial of degree~$d$, which has the following two representations:
\[
  i_Q \ = \ i_Q(z) \ = \ \sum_{j=0}^d c_j z^j \ = \
  \sum_{i=0}^d a_i \binom{z+d-i}{d}.\label{eq:ip}
\]
Here we chose the letter $z$ for the independent variable in order to
emphasize that we think of~$i_Q$ as a polynomial defined over the
complex numbers.  The coefficients $c_0$, $c_{d-1}$ and $c_d$ in the
first representation are positive, while the others generally can
vanish or take on either sign. In contrast, a famous theorem of
Stanley \cite{Stanley80} asserts that all coefficients~$a_i$ of~$i_Q$
in the latter representation are nonnegative, $a_i\ge0$ for $0\le i\le
d$.

Such nonnegativity information is also available for other
combinatorially defined polynomials, a case in point being the
chromatic polynomial of a graph (cf. Proposition~\ref{prop:chromatic}
below). An early example of how combinatorial information might be
gleaned from studying roots of such polynomials is the
\emph{Birkhoff--Lewis Conjecture}, which asserts that no chromatic
polynomial has a root in the real interval $[4,\infty)$. Somewhat
ironically, even though it was formulated as a new inroad towards
settling the Four Color Conjecture (which it implies), the latter is
now a Theorem, while the former is still open.  Nevertheless, since at
least 1965~\cite{Hall-Siry-Vanderslice1965}, the complex roots of
chromatic polynomials have received close scrutiny. A well-known
recent result by Sokal~\cite{Sokal2004} states that their complex
roots are dense in the entire complex plane, if one allows
arbitrarily large graphs. He was motivated by applications in physics
to the \emph{Potts model partition function}.

Coming back to Ehrhart polynomials, first bounds obtained in
\cite{Beck-etal05} on the location of the roots of~$i_Q$ for fixed $d$
were
substantially improved by Braun~\cite{Braun06} and Braun \&
Develin~\cite{Braun-Develin06}. All of these papers use the
nonnegativity of the $a_i$'s, but Braun's crucial new insight is to
think of the value $i_Q(z)$ at each $z\in\CC$ as a linear combination
with nonnegative coefficients of the $d+1$~complex numbers
$b_i=b_i(z)=\binom{z+d-i}{d}$. In particular, for $z_0$~to be a zero
of $i_Q$, there must be a nonnegative linear combination of
the~$b_i(z_0)$ that sums to zero.

In this paper, we extend and generalize Braun's bounds on the location
of roots for the binomial coefficient basis. We propose a unified
approach using \emph{Gale duality} to bound the location of roots, that

\begin{compactitem}
\item works in exactly the same way for \emph{all} bases of the vector
  space $P_d$ of polynomials of degree at most~$d$
  (Theorem~\ref{thm:roots1}), and

\item allows one to incorporate arbitrary additional linear equations
  and inequalities between the coefficients~$a_i$ beyond mere
  nonnegativity (Theorem~\ref{thm:ineqs}). This is applied in
  Section~\ref{sec:applications} to the case of Ehrhart and chromatic
  polynomials (Figures~\ref{fig:plotd10}~and~\ref{fig:ChrEq}).
\end{compactitem}

We apply our approach in Section~\ref{sec:bounds} to explicitly bound
the location of the roots of polynomials with nonnegative coefficients
with respect to four common bases of $P_d$; the detailed treatment of
the binomial coefficient basis comprises
Section~\ref{sec:binomial}. Throughout, we focus on bounding the
location of the non-real roots, as the case of real roots is much more
straightforward (Observation~\ref{obs:realroots}).

\medskip In Section~\ref{sec:random}, we use our method to explain why
the roots of ``random'' polynomials with nonnegative coefficients (for
a suitable meaning of ``random'') tend to clump together, by tracing
this behavior back to properties of the basis polynomials
(Figures~\ref{fig:contour}~and~\ref{fig:contourother}).

\subsection{Sketch of the method}

Let $B=\{b_0,\dots,b_d\}$ be any basis of~$P_d$, the
$(d+1)$-dimensional vector space of real polynomials of degree at
most~$d$ in one variable.

\begin{compactitem}
\item We regard $B$ as a collection of \emph{vector fields}: for each
  complex number $z\in\CC$, the basis elements $b_0(z),\dots, b_d(z)$
  define a configuration $\cB(z)=\big(w_0(z),\dots,w_d(z)\big)$ of
  real vectors $w_j(z)=(\Real b_j(z), \Imag b_j(z))^T$ in the plane
  $\RR^2$. This point of view converts the algebraic problem of
  bounding the location of roots of a polynomial into a combinatorial
  problem concerning the discrete geometry of vector configurations.
  
\item We analyze the combinatorics of 
  $\cB(z)$ in terms of  the Gale dual 
  configuration~$\cB^*(z)$. In particular, there exists a polynomial
  $f=\sum_{i=0}^d a_i b_i(z)$ with nonnegative coefficients $a_i\ge0$
  and a root at $z=z_0$ whenever the vector
  configuration~$\cB(z_0)$ has a nonnegative circuit, and this occurs
  whenever $\cB^*(z_0)$~has a nonnegative cocircuit.

  The important point is that we obtain a semi-explicit expression
  for~$\cB^*$ for \emph{any} basis of~$P_d$, not just the binomial
  coefficient basis. In fact, for the power basis $b_i=z^i$, the
  rising and falling factorial bases $b_i=z^{\overline i},
  z^{\underline i}$, and the binomial coefficient basis
  $b_i=\binom{z+d-i}{d}$ we can make the Gale dual completely
  explicit.

\item In concrete situations one often has more information about~$f$. Gale duality naturally
  allows to incorporate any linear equations and inequalities on the
  coefficients, and in some cases this leads to additional
  restrictions on the location of roots.

\item
  As an illustration, we show how the inequality $a_d\le a_0+a_1$ that
  is valid for Ehrhart polynomials further constrains the location of
  the roots of~$i_Q$. We also study the case of chromatic polynomials,
  for which Brenti~\cite{Brenti92} has shown the nonnegativity of the
  coefficients with respect to the binomial coefficient basis.

\item 
  Braun~\&~Develin~\cite{Braun-Develin06} derive an implicit equation
  for a curve $\mathcal{C}$ bounding the possible locations of roots
  of $f=\sum_{i=0}^d a_i\binom{z+d-i}{d}$, and our method gives an
  explicit equation for a real algebraic curve whose outermost oval is
  precisely~$\mathcal{C}$.

\end{compactitem}

It is instructive to visualize the vector fields $w_0,\dots,w_d$ for
the binomial coefficient basis, i.e., when 
$b_j(z)=\binom{z+d-j}{d}=R_j(z)+i I_j(z)$; recall that
$w_j(z)=\big(R_j(z),I_j(z)\big)^T$.

\begin{figure}[htbp]
  \centering
  \includegraphics[width=.55\linewidth]{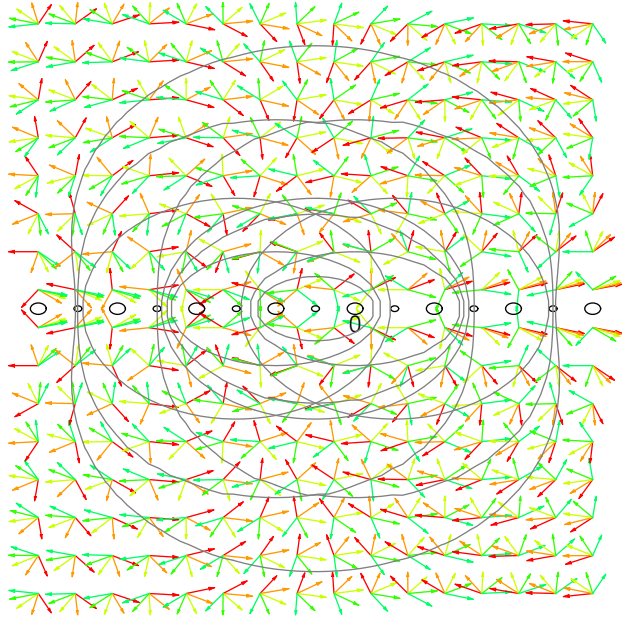}
  \caption{The values of $\{\binom{z+d-i}{d}:0\le i\le d\}$ at
    different points in the complex plane, for $d=6$. All vectors are
    normalized to the same length. In gray, the locus of points where
    two vectors become collinear.}
  \label{fig:flowers6}
\end{figure}

From Figure~\ref{fig:flowers6}, it appears that at points far away
from the origin the vectors~$w_i$ are all ``acute'', i.e., contained
in a half-plane (that varies from point to point), while closer
to the origin they positively span the entire space. If true in
general, this would imply that \emph{far away from the origin, $f$
  cannot have any roots}.

The detailed analysis (and proof) of this observation will take up the
bulk of the paper, Sections~\ref{sec:gale}~to~\ref{sec:applications},
and in this special case may be summarized as follows:

\begin{theoremstar}
  Let $d$ be a positive integer and $\mathcal Z_d$ the set of complex,
  non-real numbers that are zeros of non-identically vanishing polynomials
  of the form 
  \[
    f(z) \ = \ \sum_{j=0}^d a_j \binom{z+d-j}{d},
  \] 
  with $a_j\ge0$ for $j=0,\dots,d$. Then $\mathcal Z_d$ is the set of
  non-real points in the region
  bounded by the outermost oval of the real algebraic curve of degree
  $d-1$ in the complex plane with equation
  \[
    \frac{\binom{z}{d}\overline{\binom{z+d}{d}}
      - \overline{\binom{z}{d}}\binom{z+d}{d}}{z-\bz}
    \ = \
    0,
  \]
  where $\bar\cdot$ denotes complex conjugation.  This bound is tight,
  in the sense that any point inside~$\mathcal Z_d$ is a root of some
  such~$f$. Moreover, there is an explicit representation of this
  equation as the determinant of a tridiagonal matrix;
  see~\eqref{eq:tridiagonal} and Proposition~\ref{prop:D}.

  The real roots of any such $f$ all lie in the real interval $[-d,d-1]$.
\end{theoremstar}

From contemplating Figure~\ref{fig:flowers6}, a naive strategy for
bounding the locations of the roots comes to mind: First, try to prove
that for ``far away''~$z$ the $w_i(z)$ positively span a convex
pointed $2$-dimensional cone~$\tau$. Then determine the generators
$w_k(z)$, $w_l(z)$ of its facets, and the locus $\mathcal C$ of all
$z\in\CC$ for which these facet vectors ``tip over'', i.e., become
collinear. By continuity, for $z_0$ inside $\mathcal C$ the origin is
a nonnegative linear combination of the $w_i$, and thus $z_0$ is a
possible root.

The alert reader will perhaps have lost track of even the
\emph{number} of holes in this argument! As a sample, it is a priori
not clear (but true, at least for the binomial coefficient basis) that
the $w_i(z)$ in fact span a pointed cone for \emph{all}~$z$ of large
enough absolute value. It is even less clear (but true in this case) that the
vectors spanning facets of~$\tau$ far away from the origin will still
define facets just before $\tau$ ceases to be convex closer to the
origin.  Furthermore, the locus $\mathcal C$ might (and does) have
multiple components, suggesting that one has to exercise more care
when talking about points $z_0$ ``inside'' $\mathcal C$.

However, the real problem with this approach lies with the fact that
the locus of collinearity of $w_i(z)$ and $w_j(z)$ is the vanishing
locus of the determinant
$
  \Delta_{ij} =  \left|
    \begin{smallmatrix}
      R_i & R_j\\
      I_i & I_j
    \end{smallmatrix}
  \right|,
$
and evaluating this polynomial explicitly quickly becomes a daunting
task; moreover, it is not at all clear how the knowledge of
$\Delta_{ij}$ for any particular basis would help for other bases
of~$P_d$. 

We now present our method that overcomes all these obstacles.

\section{Gale duality}\label{sec:gale}

\subsection{Overview}

Consider a polynomial $f=\sum_{i=0}^d a_i b_i$ of degree~$d$, expanded
with respect to a basis $B=\{b_0,\dots,b_d\}$ of $P_d$, the vector
space of all polynomials in one complex variable of degree at
most~$d$. For the moment, we will focus on the complex, non-real roots
of~$f$. To find these, rewrite the real and complex parts
of the condition $f(z) = 0$ in the form
\begin{equation}
  \label{eq:primal}
  \begin{pmatrix}
    R_0 & R_1 & \dots & R_d\\
    I_0 & I_1 & \dots & I_d
  \end{pmatrix}
  \begin{pmatrix}
    a_0\\ a_1\\\vdots\\a_d
  \end{pmatrix}
  \ = \ 0,
\end{equation}
where $R_j=R_j(x,y)$ and $I_j=I_j(x,y)$ stand for the real and
imaginary parts of the polynomial $b_j(x+iy)$.

As suggested in the Introduction, we now regard each basis element
$b_i$ not as a complex polynomial, but as a real vector
$w_i(x,y)=(R_i,I_i)^T\in\RR^2$. Then there exists some polynomial~$f$
with a root at $z=x+iy$ if and only if there exist real coefficients
$a_0,\dots,a_d$ with 
\[
  \sum_{i=0}^d a_i w_i(x,y) \ = \ 0.
\] 
If we impose the additional restriction that the $a_i$ be nonnegative
but not all zero, this is only possible if the positive span of the
$w_i$ includes the origin. Among all such linear combinations summing
to zero, we now consider only \emph{support-minimal} ones, i.e., those
with the minimum number of nonzero coefficients~$a_i$. In oriented
matroid terminology, the ordered collection~$\sigma$ of signs of the
coefficients of such a support-minimal linear combination is called a
\emph{circuit} of the (full-dimensional) vector configuration
$W=(w_0,\dots,w_d)\subset\RR^2$. To proceed, we regard $W$ as a
$2\times (d+1)$-matrix.  A \emph{Gale dual} vector configuration
$\oW=(\ow_0,\dots,\ow_d)\subset\RR^{d-1}$ of $W$ is the ordered set of
rows of any matrix, also called $\oW$, whose columns form a basis for
the (row) kernel of the matrix $W$, so that $W\oW=0$~\cite{Ziegler98}.
Gale duality states that the collection of signs~$\sigma$~is a
\emph{cocircuit} of~$\oW$. This means that there exists a linear form
$g$ on $\RR^{d-1}$ with $(\sign g(\ow_i):i=0,\dots,d)=\sigma$.

Clearly, any circuit of $W$ has either two or three non-zero entries
(unless it is the zero circuit, which we exclude from the discussion).
Because $z_0$ is a root of $f$ if and only if there exists a
nonnegative circuit of
$W(z_0)=\big(w_0(z_0),\dots,w_d(z_0)\big)\subset\RR^2$, by Gale
duality this happens if and only if there exists a cocircuit of
$\oW(z_0)=\big(\ow_0(z_0),\dots,\ow_d(z_0)\big)\subset\RR^{d-1}$, i.e.,
if and only if there exists a linear form on $\RR^{d-1}$ that vanishes
on all of the $\ow_i$ except for either two or three of them, and on
those evaluates to the same sign. Geometrically, there must exist a
linear hyperplane in $\RR^{d-1}$ that contains all vectors $\ow_i$
except for two or three, and has those on the same side.

Thus, we have traded the search for the locus of two collinear vectors
among the $w_i\in\RR^2$ (a problem involving only two pieces of input
data) for the task of finding a Gale dual $\oW$ in the much
higher-dimensional space $\RR^{d-1}$, and hyperplanes passing through
almost all of the~$\ow_i$~---~a problem involving almost the entire
input!

That this is not crazy, but instead effective, is explained by the
fact that passing to the higher-dimensional representation is possible
in great generality, and moreover greatly simplifies the structure of
the problem; see Proposition~\ref{prop:dualmatrix} below.

\subsection{Implementation}

Let $B=\{b_i:0\le i\le d\}$ be any basis of~$P_d$.

\subsubsection{The Gale dual}

Form the matrix
\[
   W \ = \ W(x,y) \ = \   
   \begin{pmatrix}
     R_0 & R_1 & \dots & R_d\\
     I_0 & I_1 & \dots & I_d
   \end{pmatrix},
\]
where $R_j=R_j(x,y)$ and $I_j=I_j(x,y)$ denote the real and imaginary
part of the complex polynomial
$b_j=b_j(x+iy)$. 
The rank of $W$ is $2$, 
so any Gale dual matrix $\oW$
to~$W$ has size $(d+1)\times(d-1)$. The following proposition gives an
explicit representative for~$\oW$ involving polynomials $p_k,q_k,r_k$
that depend on the basis~$B$. For four especially relevant bases, we
will make the Gale dual $\oW$ completely explicit. These bases are:
\begin{itemize}
\item The \emph{power basis}, where $b_i=z^i$;
\item the \emph{falling factorial basis}, where $b_i=z^{\underline
    i}=z(z-1)\cdots(z-i+1)$;
\item the \emph{rising factorial basis}, where
  $b_i=z^{\overline{i}}=z(z+1)\cdots(z+i-1)$; and
\item the \emph{binomial coefficient basis}, where $b_i=\binom{z+d-i}{d}$.
\end{itemize}
Here $z^0=z^{\underline 0}=z^{\overline 0}=1$. 

\begin{proposition}\label{prop:dualmatrix}
  A Gale dual matrix to $W$ may be chosen to have exactly three non-zero
  diagonals
  \begin{equation}\label{eq:dualmatrix}
     \oW \ = \ \oW(x,y) \ = \
     \begin{pmatrix}
       \phantom{+}p_0 & 0 & 0 & \dots & 0\\
       -q_0 & \phantom{+}p_{1} & 0 & \dots & 0\\
       \phantom{+}r_0 & -q_{1} & \ddots&&\vdots\\
       0 & \phantom{+}r_1 & \ddots & \ddots\\
       \vdots && \ddots&\ddots& \phantom{+}p_{d-2}\\
       &&& r_{d-3} & -q_{d-2}\\
       0 &\dots&& 0 & \phantom{+}r_{d-2}
     \end{pmatrix}.
  \end{equation}
  Moreover, its entries may be chosen to lie in $\RR[x,y]$. For the
  four bases considered, we may choose the following explicit values:
  \begin{center}\renewcommand{\arraystretch}{1.2}
    \begin{tabular}[b]{c||c|c|c} 
      $b_i$ & $p_k$ & $q_k$ & $r_k$\\\hline\hline
      $z^i$ & $x^2+y^2$ & $2x$ & $1$\\
      $z^{\underline i}$ & $(x-k)^2+y^2$ & $2(x-k)-1$ & $1$\\
      $z^{\overline i}$ & $(x+k)^2+y^2$ & $2(x+k)+1$ & $1$\\
      $\binom{z+d-i}{d}$ & $(x-k)^2+y^2$ & $p_k+r_k-d(d-1)$ &
      $p_{k+1-d}$  
    \end{tabular}.
  \end{center}
  Note that in the last row,
  $q_k=2\big(x-(k-\tfrac{d-1}{2})\big)^2+2y^2-\tfrac{d^2-1}{2}$. 
\end{proposition}

\begin{proof}
  We first prove that the matrix $\oW$ can be chosen to have the
  displayed triple band structure regardless of the basis $B$ chosen
  for $P_d$. For this, define the rational functions
  $g_k=\frac{b_{k+1}}{b_k}\in\RR(z)$ for $0\le k\le d-1$; specific
  values for $g_k$ become apparent from the relations $z^{k+1}=z\cdot
  z^k$, $z^{\underline{k+1}}=(z-k)z^{\underline{k}}$,
  $z^{\overline{k+1}}=(z+k)z^{\overline{k}}$ and
  $\binom{z+d-k-1}{d}=\frac{z-k}{z+d-k}\binom{z+d-k}{d}$.
  
  The triple $(p_k, q_k, r_k)$ lists nontrivial coefficients of a real
  syzygy
  \[
     p_k b_k + q_k b_{k+1} + r_k b_{k+2} \ = \ 
     b_k\big(p_k + g_k q_k + g_k g_{k+1} r_k \big) \ = \ 0
  \]
  whenever 
  \[
  \begin{pmatrix}
    1 & \Real g_k & \Real g_k g_{k+1} \\
    0 & \Imag g_k & \Imag g_k g_{k+1}
  \end{pmatrix}
  \begin{pmatrix}
    p_k \\ q_k \\ r_k
  \end{pmatrix} \ = \ 
  \begin{pmatrix}
    0\\0
  \end{pmatrix}.
  \]
  But the displayed matrix with entries in $\RR(x,y)$, call it $M$,
  obviously has rank at least~$1$, and rank~$2$ whenever $\Imag
  g_k(x+iy)\ne0$, so that such triples certainly exist. Moreover, by
  multiplying with a common denominator we may assume
  $p_k,q_k,r_k\in\RR[x,y]$, and so the relations $p_k b_k + q_k
  b_{k+1} + r_k b_{k+2}=0$ imply that $\oW$ is in fact a Gale dual of
  $W$.  The concrete syzygies listed above arise by choosing 
  explicit bases for $\ker M$.
\end{proof}

\begin{remark}
  Another interesting case is that of polynomials with symmetric
  coefficients. For instance, if $f=\sum_{i=0}^d a_i\binom{z+d-i}{d}$
  and $a_i=a_{d-i}$, we may expand $f$ in the basis
  $B=\big\{\binom{z+d-i}{d}+\binom{z+i}{d}:0\le i\le
  \fl{\frac{d}{2}}\big\}$ of the vector space of polynomials with
  symmetric coefficients in the binomial coefficient basis. However,
  the coefficients of syzygies of these $b_k$ do not appear to be as
  simple as the ones listed in Proposition~\ref{prop:dualmatrix}. For
  example, a typical coefficient (namely, $q_1$ for $d=8$) reads
  \[
      -8
      \big(((x+\alpha_1)^2+y^2)((x+\alpha_2)^2+y^2)+\gamma_1\big)
      \big(((x+\beta_1)^2+y^2)((x+\beta_2)^2+y^2)+\gamma_2\big)+\gamma,
  \]
  where $\alpha_1,\alpha_2$ are the roots of
  $\alpha^2-\alpha+\rho_1=0$ (so that $\alpha_1+\alpha_2=1$),
  $\beta_1,\beta_2$ are the roots of $\beta^2-\beta+\rho_2=0$,
  $\rho_1,\rho_2$ are the roots of
  $\rho^2-\frac{29}{2}\rho-\frac{231}{8}=0$,
  $\gamma_1+\gamma_2=\frac{135}{4}$,
  $\gamma_1=\frac{135}{8}(1-\frac{61}{\sqrt{649}})$, and
  $\gamma=\frac{874800}{649}$. We will not pursue this basis further
  in this paper.
\end{remark}

\subsubsection{The determinants}

Recall that two vectors $w_j(z),w_k(z)$ become collinear at some point
$z\in\CC$ whenever there exists a circuit of the vector configuration
$W(z)$ with exactly two non-zero entries. By Gale duality, this means that the
Gale dual vector configuration $\oW(z)$ has a cocircuit with support
$2$, i.e., the determinant of the matrix obtained by deleting two rows
from $\oW$ vanishes. Our approach rests on the fact that we can give
fairly explicit expressions for these determinants for the four bases
considered here.

\begin{lemma}\label{lem:det}
  Let $\oW_{(j,k)}=\oW_{(j,k)}(x,y)$ denote the square matrix obtained
  by deleting rows $j$ and $k$ from $\oW$, where $0\le j < k\le d$ (so
  that we number the rows from $0$ to $d$). Then
  \[
     \det\oW_{(j,k)} \ = \ p_0\cdots p_{j-1} D_{j,k} r_{k-1}\cdots r_{d-2},
  \]
  where $D_{j,k}=D_{j,k}(x,y)$ is the determinant of the tridiagonal matrix
  \begin{equation}\label{eq:tridiagonal}
     \begin{pmatrix}
       -q_j & \phantom{+}p_{j+1} & 0 & \dots & 0\\
       \phantom{+}r_j & -q_{j+1} & \ddots&&\vdots\\
       0 & \phantom{+}r_{j+1} & \ddots & \ddots\\
       \vdots && \ddots&\ddots& \phantom{+}p_{k-2}\\
       &&& r_{k-3} & -q_{k-2}\\
     \end{pmatrix}.
  \end{equation}
  Here $D_{j,j+1}:=1$, and the leading resp.\ trailing
  products are $1$ if $j=0$ resp.\ $k=d$. In particular, $D_{j,j+2}=-q_j$.
\end{lemma}

\begin{proof}
  The matrix $\oW_{(j,k)}$ decomposes into three blocks, whose
  determinants yield the stated formula, and two additional elements
  $r_{j-1}$ and $p_{k-1}$ that do not contribute to $\det\oW_{(j,k)}$.
\end{proof}

\begin{proposition}\label{prop:D}
  Set $z=x+iy$ and $\bz=x-iy$. Then
  \[
    D_{j,k}(x,y) \ = \
    \frac{(-1)^{k-j-1}}{z-\bz}\big( f_{j,k}(z)-f_{j,k}(\bz) \big),
  \]  
  where the polynomials $f_{j,k}(z)$ are given in the following table:
  \begin{center}
    \begin{tabular}[c]{c|c}
      $b_i$ & $f_{j,k}(z)$\\\hline
      $z^i$ & $z^{k-j}$ \\
      $z^{\underline{i}}$ & $(z-k+1)\cdots(z-j)$ \\
      $z^{\overline{i}}$ & $(z+k-1)\cdots(z+j)$ \\
      $\binom{z+d-i}{d}$ & 
      $\frac{1}{d}(z-k+1)\cdots(z-j)(\bz+d-k+1)\cdots(\bz+d-j)$
    \end{tabular}
  \end{center}
  The $D_{j,k}$ are real polynomials 
  with even degrees in $y$.
\end{proposition}

\begin{proof}
  It is well known that the determinant $D_n$ of an $n\times n$
  tridiagonal matrix $A=(a_{ij})$ satisfies the three-term recursion
  relation $D_n=a_{nn}D_{n-1}-a_{n,n-1}a_{n-1,n}D_{n-2}$. Solving this
  recursion for the matrix from Lemma~\ref{lem:det} with the values
  from Proposition~\ref{prop:dualmatrix} and the boundary conditions
  $D_{j,j+1}=1$ and $D_{j,j+2}=-q_j$ yields the stated expressions.
\end{proof}

\subsection{The real case}

Up to now, we have only considered complex, non-real roots of $f$. The case of
real roots is much simpler, and the machinery used for complex roots
specializes in a straightforward way to the real case.  If we regard
both $f$ and the $b_i$ as polynomials in one real variable, the matrix
$W=W^\RR$ reduces to the single row $W^\RR=(b_0,\dots,b_d)$.

\begin{proposition} \label{prop:realkernel}
A basis for the kernel of $W^\RR$ is given by the columns of the matrix
\[
   \oW^\RR
   \ = \
   \begin{pmatrix}
     -b_1 & 0 & 0 & \dots & 0 \\
     b_0 & -b_2 & 0 & \dots & 0 \\
      0 & b_1 & \ddots &  & \vdots \\
      \vdots & 0   & \ddots & \ddots & 0 \\
      \vdots & \vdots & & \ddots & -b_d\\
      0 & \dots & & 0 & b_{d-1}
   \end{pmatrix}
\]
of size $(d+1)\times d$. The determinant of the matrix obtained by
deleting row~$j$ from~$W$ is
\[
  \det\oW_j
  \ = \
  (-1)^j b_j \Pi 
  \qquad \text{for } j=0,1,\dots, d,
\]
with $\Pi=b_1b_2\cdots b_{d-1}$.\hfill$\Box$
\end{proposition}

\section{Bounding the location of roots}\label{sec:bounds}

We first treat the case of complex, non-real roots.  For each ordered
triple of indices $i,j,k$ with $0\le i<j<k\le d$, denote by
$H_{i,j,k}$ the hyperplane in $\RR^{d-1}$ spanned by the rows of the
matrix $\oW_{(i,j,k)}$, obtained by deleting the rows
$\ow_i,\ow_j,\ow_k$ from $\oW$.

\begin{definition}
  $\cS_{i,j,k}$ is the set of all $z=x+iy\in\CC$ such that
  $H_{i,j,k}=H_{i,j,k}(x,y)$ induces a nonnegative cocircuit, i.e.,
  the vectors $\ow_i=\ow_i(x,y),\ow_j=\ow_j(x,y),\ow_k=\ow_k(x,y)$ all
  (weakly) lie on the same side of~$H_{i,j,k}(x,y)$.
\end{definition}

The sets $\cS_{i,j,k}$ are crucial for our purposes for the following
reason: If $z\in\cS_{i,j,k}$, then the corresponding Gale primal
vectors $w_i,w_j,w_k$ form a nonnegative circuit, and thus yield a
nonnegative combination of all $w$'s that sums to zero; in other
words, there exists some polynomial~$f$ with nonnegative coefficients
in the chosen basis $B$ that has a zero at~$z$.  On the other hand, if
$z\notin\cS_{i,j,k}$, we can only conclude that the three particular
Gale primal vectors $w_i,w_j,w_k$ do \emph{not} form a circuit, and so
are not responsible for the possible zero $z$ of $f$.

\begin{proposition}\label{prop:signs}
  For $0\le i<j<k\le d$ and $z\in\CC\smallsetminus\RR$, let
  $\sigma_{i,j,k}(z)$ be the set of signs
  \begin{eqnarray*}
    &&
    \left\{(-1)^i \sign\det\oW_{(j,k)}(x,y),\ 
      (-1)^{j-1} \sign\det\oW_{(i,k)}(x,y),\ 
      (-1)^{k-2} \sign\det\oW_{(i,j)}(x,y)
    \right\}    
    \\
    &=&
    \left\{
      (-1)^i \sign D_{j,k}(x,y),\ 
      (-1)^{j-1} \sign D_{i,k}(x,y),\ 
      (-1)^{k-2} \sign D_{i,j}(x,y)
    \right\}    .
  \end{eqnarray*}
  Then each $\cS_{i,j,k}\subset\RR^2$ is a
  semialgebraic set defined as the locus of all $(x,y)$ such that
  \[
    \{\pm1\} \ \not\subset\ 
    \sigma_{i,j,k}(x+iy).
  \]
\end{proposition}

\begin{proof}
  We obtain a linear form $\varphi_{i,j,k}$ on $\RR^{d-1}$ whose
  vanishing locus is the hyperplane $H_{i,j,k}$ by adding a first row
  of variables $x_1,\dots,x_{d-1}$ to $\oW_{(i,j,k)}$ and expanding
  the determinant of that square matrix along the first row. The value
  of $\varphi_{i,j,k}$ on $\ow_i$, say, is given by the sign~$(-1)^i$
  of the permutation that interchanges rows $0$~and~$i$ in the matrix
  $\oW_{(j,k)}$, times $\det \oW_{(j,k)}$.  For $H_{i,j,k}$ to define
  a (positive or negative) cocircuit, the signs obtained in this way
  for $\ow_i$, $\ow_j$~and~$\ow_k$ must agree. Finally, by
  Lemma~\ref{lem:det} the signs of $\det\oW_{(j,k)}$ and $D_{j,k}$
  agree except perhaps on the real axis (on the vanishing locus of the
  $p_k$'s and $r_k$'s), and we may assume that $j-i\ge2$ and
  $k-j\ge2$.
\end{proof}

  In summary:

\begin{theorem}\label{thm:roots1}
  Let $f$ be a polynomial of degree $d$ with nonnegative coefficients
  with respect to some basis of the vector space $P_d$.  Then the set
  of non-real roots of $f$ is contained in the union of the
  semialgebraic sets $\cS_{i,j,k}$, for $0\le i<j<k\le d$. Put
  differently, if
  \[
     \{-1,1\}\ \subseteq\ 
     \big\{(-1)^i D_{j,k}(z_0),\
     (-1)^{j+1}D_{i,k}(z_0),\
     (-1)^{k}D_{i,j}(z_0)\big\}
  \]
  for each triple $(i,j,k)$ with $0\le i<j<k\le d$, then $z_0$ is not
  a root of $f$. \hfill $\Box$
\end{theorem}

After a short discussion of the real case, we will apply this result
to our four representative bases. We only discuss the power basis and
binomial coefficient basis in any detail, as the procedure for the
rising and falling factorial bases is almost exactly the same.

\subsection{The real case}

A real number $x\in\RR$ is a real root of some polynomial $f$ with
nonnegative coefficients with respect to a fixed basis if and only if
either $x$~is a root of some basis polynomial, or two basis
polynomials differ in sign when evaluated at~$x$. In other words:

\begin{observation}\label{obs:realroots}
  Let $f$ be a polynomial of degree $d$ with nonnegative coefficients
  with respect to some basis $\{b_0,\dots,b_d\}$ of $P_d$. Then the
  locus of possible real roots of~$f$ is the set of  $x\in\RR$ for which
  $b_i(x)b_j(x)\le0$ for some $i\ne j$.\hfill $\Box$
\end{observation}

For the sake of completeness, and in response to the query of one of
the referees, we briefly rederive this result using our framework of
Gale transforms.

\begin{proof}
  In complete analogy to the complex case, denote for $0\le i<j\le d$
  by $\oW^\RR_{(i,j)}$ the matrix obtained by deleting rows
  $w_i$~and~$w_j$ from~$\oW^\RR$, by $H_{i,j}$ the hyperplane in
  $\RR^d$ spanned by the rows of $\oW^\RR_{(i,j)}$, and by
  $\cS_{i,j}\subseteq\RR$ the set of all $x\in\RR$ such that
  $H_{i,j}(x)$ induces a positive cocircuit, i.e., the vectors
  $\ow_i=\ow_i(x)$ and $\ow_j=\ow_j(x)$ lie on the same side of
  $H_{i,j}$.

  To find these cocircuits explicitly, build a linear form
  $\varphi_{i,j}$ on $\RR^d$ that defines~$H_{i,j}$ by adding a first
  row $(x_1,\dots,x_d)$ of variables to~$\oW^\RR_{i,j}$ and expanding
  the determinant of that square matrix along the first row.  Just as
  in the proof of Proposition~\ref{prop:signs}, $\varphi_{i,j}(\ow_i)
  = (-1)^i \det \oW^\RR_i=b_i\Pi$ and $\varphi_{i,j}(\ow_j) =
  (-1)^{j-1}\det \oW^\RR_j=-b_j\Pi$ by
  Proposition~\ref{prop:realkernel}. In consequence, $\cS_{i,j}$~is
  the locus of points $x\in\RR$ such that $b_i(x)$ and $b_j(x)$ differ
  in sign. This finishes the proof.
\end{proof}

\subsection{The power basis}\label{subsec:power}

For $b_i=z^i$, we set $n=k-j-1$, write $D_n$ for $D_{j,k}$, and
substitute $z=re^{i\theta}$ into $D_n$:
\[
   D_n \ = \ 
   (-1)^{n}\frac{z^{n+1}-\bz^{n+1}}{z-\bz} \ = \
   (-1)^{n}r^{n}e^{-in\theta}\,
   \frac{e^{2i(n+1)\theta}-1}{e^{2i\theta}-1}.
\]
This vanishes iff $\theta=\pi l/(n+1)$ for integer $l$ with $1\le l\le
2n+1$ and $l\ne n+1$. The zero locus of $D_n$ thus consists of
$n$~lines through the origin, the ones closest to the $x$-axis having
angles $\theta=\pm \frac{\pi}{n+1}$. We conclude that $D_n$~has the same
sign throughout the entire open sector $Z_{n+1}=\{z\in\CC:-\frac{\pi}{n+1}<\arg
z<\frac{\pi}{n+1}\}$.  By substituting a positive, real value of $z$ into
$D_n=(-1)^{n}\sum_{j=0}^{n}z^j\bz^{n-j}$, we determine this sign
to be~$(-1)^{n}=(-1)^{k-j-1}$.

For $z\in Z_d$ and $0\le i<j<k\le d$ the set of signs of the
polynomials in Proposition~\ref{prop:signs} is
\[
   \sigma_{i,j,k}(z) \ = \
   \big\{
   (-1)^{i+j+k+1},\ 
   (-1)^{i+j+k},\ 
   (-1)^{i+j+k-1} \big\}
  \ = \ \{\pm1\}. 
\]
This implies $\cS_{i,j,k}=\emptyset$, and thus
Theorem~\ref{thm:roots1} recovers the classical result that a
polynomial of degree~$d$ with positive coefficients in the power basis
has no zeros in~$Z_d$; of course, this includes the case of real
roots.

\subsection{Rising and falling factorial basis} 

In both cases, the polynomials $f_{j,k}$ from Proposition~\ref{prop:D}
have the form $f_{j,k}(z)=\prod_{i=1}^{n+1} (z-a_i)$, with $n=k-j-1$
and $a_i=j-1+i$ for the falling powers and $a_i=-(j-1+i)$ in the case
of the rising powers. The transform $z\mapsto z\pm
\frac{j+k-1}{2}$ remedies this asymmetry, where we choose the
`$-$'~sign for $b_i=z^{\underline{i}}$ and the `$+$'~sign for
$b_i=z^{\overline{i}}$. The $a_i$ then become
integers or half-integers in the range $\pm\frac{n}{2}$.

Using the same type of analysis as will be detailed in
Section~\ref{sec:binomial} for the binomial coefficient basis, one can
prove that the zero locus $\DD_{j,k}=\{z\in\CC:D_{j,k}(z)=0\}$ is
smooth everywhere, that one component intersects the real axis between
each pair of adjacent $a_i$'s, and that far away from the origin
$\DD_{j,k}$ approaches the arrangement of lines through the origin
with slopes $\pm\frac{1}{n+1}, \dots, \pm\frac{n}{n+1}$;
cf.~Figure~\ref{fig:rising+falling}. We will not enter into the
details here, but instead treat the remaining basis in a separate
section.

\begin{figure}[htbp]
  \centering
  \includegraphics[width=.4\linewidth]{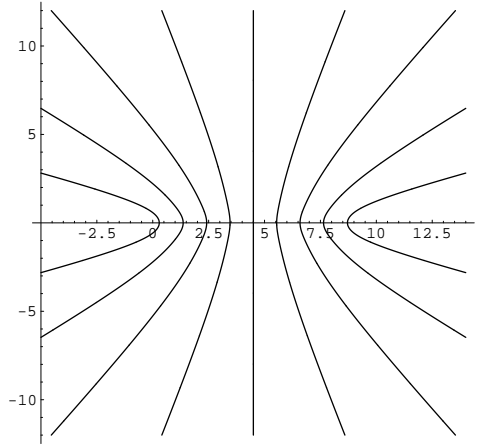}
  \caption{The locus $\DD_{0,10}$ in the case of the rising 
  factorial basis. }
  \label{fig:rising+falling}
\end{figure}

\section{The binomial coefficient basis}\label{sec:binomial}

We first get the real case out of the way: The basis polynomials all
have the same sign outside the closed interval $[-d,d-1]$, and at each
point inside this interval there are two basis polynomials that
evaluate to opposite signs. By Observation~\ref{obs:realroots},
$[-d,d-1]$ is exactly the set of possible real roots.

For the non-real roots, as before we pass to an adapted coordinate
system with respect to which the vanishing locus of $D_{j,k}$ is
centro-symmetric, by replacing
\begin{equation}
  \label{eq:coo-change}
  z
  \ \mapsto \
  z'+(k+j-d-1)/2 
\end{equation}
in $d\cdot f_{j,k}(z)$. Writing again $z$ for $z'$ yields
\[
    df_{j,k}(z) \ = \
    \prod_{i=j}^{k-1}\left(z+\frac{k+j-d-1-2i}{2}\right)
    \prod_{i=j}^{k-1}\left(\bz+\frac{k+j+d-1-2i}{2}\right).
\]
Next, we replace $i$ by $i+j$ in the first product and by $k-1-i$ in the
second, to obtain
\begin{eqnarray*}
    df_{j,k}(z) & = &
    \prod_{i=1}^{k-j}\left(z-i-\frac{\Delta}{2}\right)
    \left(\bz +i + \frac{\Delta}{2}\right),
\end{eqnarray*}
where $\Delta=d-1-k+j$. Introducing $a_i=i+\frac{\Delta}{2}$ we obtain
\begin{equation}\label{eq:symDjk}
   D_{j,k}(z) \ = \ D_{n}(z) \ = \ 
   \frac{(-1)^{n+1}}{d(\bz-z)}   \left(
     \prod_{i=1}^{n+1} (z-a_i)(\bz+a_i) - 
     \prod_{i=1}^{n+1} (\bz-a_i)(z+a_i) 
   \right),
\end{equation}
where we have set $n=k-j-1$ (so that $\Delta=d-2-n$), in accordance
with the fact that the degree of $D_{j,k}(z)$ in $z$ is $n$.

\medskip

Before examining the zero locus of $D_{j,k}(z)$, we pause to calculate
the leading coefficient. This result will be used in Section~\ref{sec:ehrhart}.

\begin{lemma}\label{lem:leading}
  The leading coefficient of $D_{j,k}(z)$ is 
  \begin{equation}
    \label{eq:leading}
    [2n] D_{j,k}(z) \ = \  (z\bz)^{k-j-1} (-1)^{k-j-1} (k-j) \ = \
    r^{2n} (-1)^n (n+1) ,
  \end{equation}
  where $z=re^{i\phi}$. It is invariant under
  substitutions of the form $z\mapsto z+z_0$, and the sign of
  $D_n(z)$ outside the outermost component of $\DD_n$ is $(-1)^n$, and
  $+1$ inside the innermost~one.
\end{lemma}

\begin{proof}
  See the Appendix.
\end{proof}

We now treat the zero locus of $D_n$. First, whenever $D_n(z)=0$,
\[
   \sum_{i=1}^{n+1} \arg(z-a_i) + \sum_{i=1}^{n+1} \arg(\bz+a_i)
   - \sum_{i=1}^{n+1} \arg(\bz-a_i) - \sum_{i=1}^{n+1} \arg(z+a_i)
   \ = \ 2l\pi
\]
for some integer $l$. Because
$\arg(z\pm a_i)=-\arg(\bz\pm a_i)$, this relation reads
\begin{equation}\label{eq:asum}
   \sum_{i=1}^{n+1} \arg(z-a_i) - \sum_{i=1}^{n+1} \arg(z+a_i) 
   \ = \ \sum_{i=1}^{n+1} \alpha_i 
   \ = \  l\pi,
\end{equation}
where $\alpha_i$ is the angle under which the segment $[-a_i,a_i]$
appears as seen from $z$ (cf.~Figure~\ref{fig:angle}). 

\begin{figure}[htbp]
  \centering
  \input{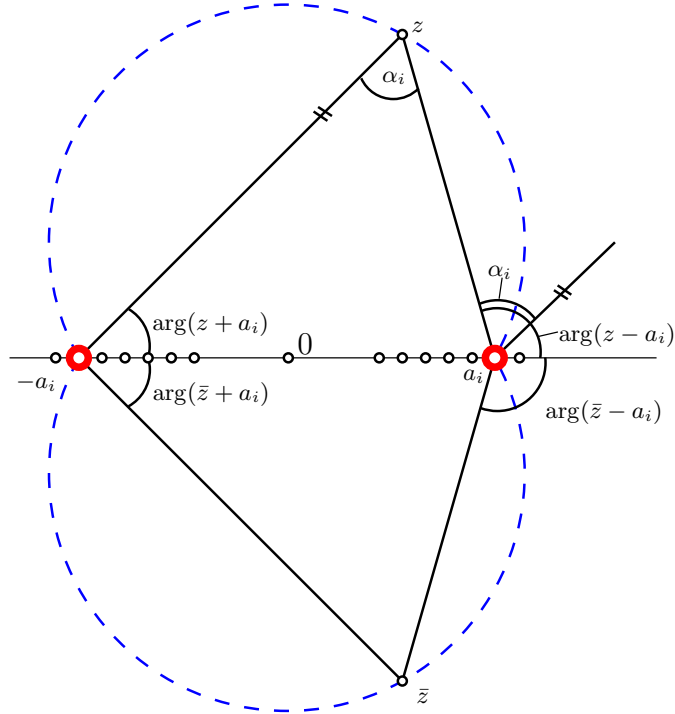}
  \caption{The segments $[-a_i,a_i]$ as seen from $z$}
  \label{fig:angle}
\end{figure}

We may assume without loss of generality that $z$ lies in the upper
half plane, and therefore that $\arg(z-a_i)>\arg(z+a_i)>0$, which
implies $l\ge 1$. On the other hand, the maximal value $(n+1)\pi$
of~\eqref{eq:asum} is achieved for real $z$ between $-a_1$~and~$a_1$,
so that $l\le n$ for non-real~$z$.  From this, we can draw several
conclusions, which we detail in Section~\ref{sec:limiting+global}. The
reader may want to just skim this material, and otherwise skip ahead
to Section~\ref{sec:bounded}, where we apply it to conclude that the
root locus is bounded.

\subsection{Limiting behavior and global geometric properties of $D_{j,k}$}
\label{sec:limiting+global}

\begin{proposition}\label{prop:circles}
  When $d$ becomes large with respect to $n$, the zero locus of
  $D_{j,k}$ approaches a union of circles passing through $\pm d/2$
  and symmetric about the imaginary axis.  For $l=1,\dots,k-j-1$,
  these circles have center
  \[
    z_{l} \ = \ -\frac{d+1-k-j}{2} - i\,\frac{d}{2}\cot\frac{l\pi}{k-j}
  \]
  and radius
  \[
    r_{l} \ = \ \frac{d}{2\sin\frac{l\pi}{k-j}}.
  \]

\end{proposition}

\begin{proof} 
  For $d$ large with respect to $n$, the points $\pm a_i$
  fuse to $\pm a=\pm\frac{d}{2}$, so that \eqref{eq:asum} reads 
  \[
    \alpha \ := \ \arg(z-a)-\arg(z+a) \ = \ \frac{l\pi}{n+1}.
  \]
  By elementary geometry, the locus of these points is a union of two
  circular arcs with the specified equations.
\end{proof}

\begin{example}
  For $n=1$, we obtain $a_1=\frac{d-1}{2}$ and
  $a_2=\frac{d+1}{2}$. For large~$d$, they approach $a=\frac{d}{2}$
  and equation \eqref{eq:asum} says $\alpha=\frac{\pi}{2}$. In that
  limit, $\DD_1$ thus approaches the circumference with center $0$ and
  radius $\frac{d}{2}$. For smaller values of $d$, directly evaluating
  equation~\eqref{eq:symDjk} yields
  \[
     D_1(z) \ = \ z\bz-a_1a_2,
  \]
  which describes a circumference of center $0$ and radius
  $\sqrt{a_1a_2\mathstrut}=\frac12\sqrt{d^2-1\mathstrut}<\frac{d}{2}$. 
\end{example}

\begin{proposition}\label{prop:smooth}
  The plane algebraic curve $\DD_{j,k}$ with equation $D_{j,k}(z)=0$
  is smooth. The only points where it has horizontal tangent vectors
  lie on the $y$-axis.
\end{proposition}

This is proved in the Appendix.

\begin{proposition}\label{prop:ovals}
  All algebraic curves $\DD_{j,k}$ consist of $n=k-j-1$ nested ovals.
  The $i$-th~oval intersects the real axis inside the union of open
  intervals $\pm(a_i,a_{i+1})$, for $i=1,\dots,n$.
\end{proposition}

\begin{proof} 
  Let $\phi\in S^1\smallsetminus S^0$ be a non-real unit vector and $\rho$
  the ray through the origin and $\phi$. At each point~$p$ of
  $\DD_{j,k}\cap\rho$, the angle sum $\sum_{i=1}^{n+1}\alpha_i$ takes on
  some value $l\pi$ among the discrete set $\{\pi,\dots,n\pi\}$, and
  therefore this value remains constant on the entire connected
  component to which $p$ belongs. The argument extends to the real
  axis by smoothness of~$\DD_{j,k}$. 

  For the second statement, observe that the value of
  $\alpha_j=\arg(z-a_j)-\arg(z+a_j)$ increases by almost~$\pi$ as $z$
  travels from $a_j+\varepsilon+i\delta $ to
  $a_j-\varepsilon+i\delta$, for $0<\delta\ll\varepsilon\ll 1$.
\end{proof}


\begin{example}
  For $D_{0,10}=D_9$ and $d=10$, we obtain the picture of
  Figure~\ref{fig:d10}.
\end{example}

\begin{figure}[htbp]
  \centering
  \includegraphics[width=.25\linewidth]{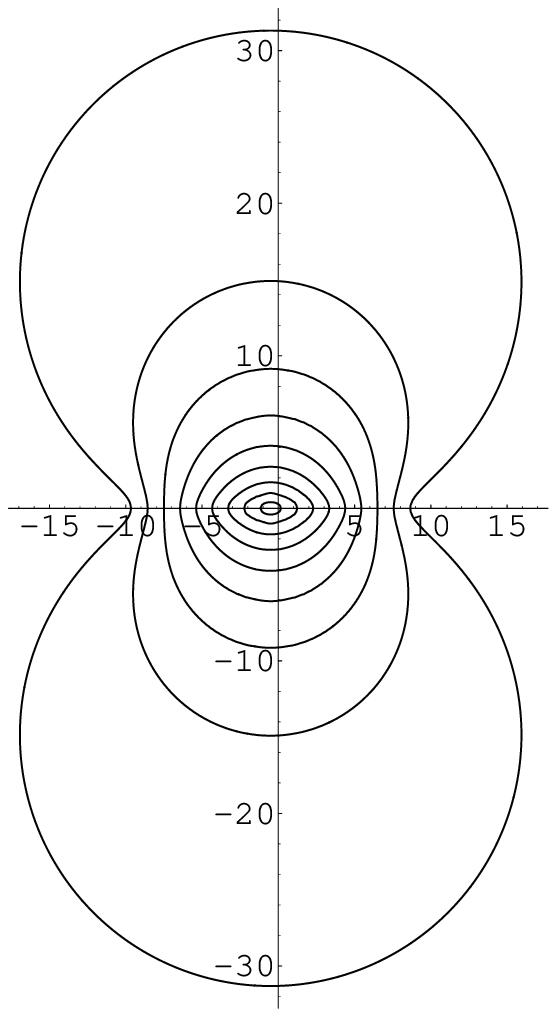} \qquad
  \includegraphics[width=.25\linewidth]{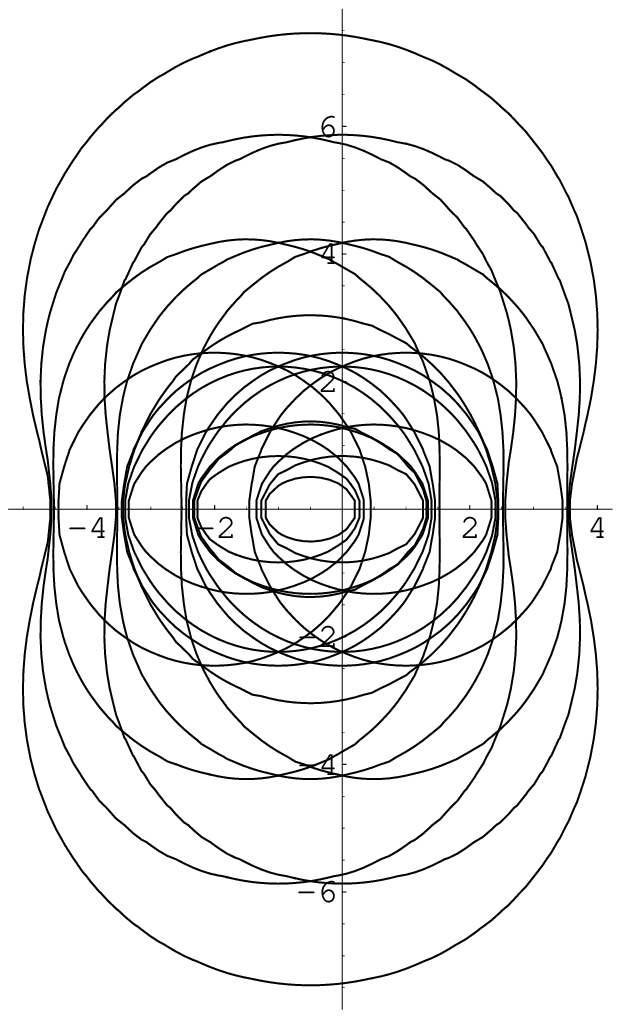}
  
  \caption{\emph{Left:} The curve $\DD_{0,10}$ for $d=10$;
    \emph{Right:} the curves $\DD_{i,j}$ with $0\le i<j\le 5$}
  \label{fig:d10}
\end{figure}

To continue, we introduce some useful notation.  By \eqref{eq:symDjk},
the formula for $D_{j,k}(z)$ involves the points
$a_{j,k;i}=i+\frac12(d-1-k+j)$ for $1\le i\le k-j$, so that
\[
  (a_{j,k;1},\dots,a_{j,k;k-j}) \ = \
  \left(\frac{d}{2}-\frac{k-j-1}{2},\dots,\frac{d}{2}+\frac{k-j-1}{2}\right).
\]
We write $\alpha(\pm a_{j,k;i};z)$ for the angle under which $z\in\CC$
sees the segment $[-a_{j,k;i},a_{j,k;i}]$, and $A(j,k;z) =
\sum_{i=1}^{k-j}\alpha(\pm a_{j,k;i};z)$ for the corresponding angle
sum. Moreover,  let 
\[
   \DD_{j,k;l} \ = \ \big\{(x,y)\in\RR^2:A(j,k;x+iy)= l\pi\big\}
   \qquad\text{for }l=1,\dots,k-j-1 
\]
be the $l$-th oval of $\DD_{j,k}$, and $\cl \DD_{j,k;l}$ the
closure of the region in $\RR^2$ bounded by~$\DD_{j,k;l}$.

\begin{remark}
  The arrangement of ovals $\{\DD_{j,k;l}:0\le j<k\le d, 1\le l\le
  k-j\}$ has several interesting combinatorial properties, which we
  will not pursue in this paper. Here we would only like to point out
  the triple points where components of $\DD_{j,r}$, $\DD_{r,k}$ and
  $\DD_{j,k}$ intersect.
\end{remark}

\begin{proposition}\label{prop:bounded}
  Let $0\le j \le j'<k'\le k \le d$ and $1\le l\le k-j-1$ be integers.

  Then $\DD_{j',k'}\subset\cl\DD_{j,k;1}\smallsetminus\cl\DD_{j,k;k-j-1}$.
  In particular, all components of all curves~$\DD_{j,k}$ are
  contained in the topological closure of
  $\cl\DD_{0,d;1}\smallsetminus\cl\DD_{0,d;d-1}$.  Moreover, for all
  integers $\delta_1,\delta_2$ with $0\le\delta_1\le j$,\
  $0\le\delta_2\le d-k$ and (where appropriate) $1+\delta_1+\delta_2
  \le l \le d-\delta_1-\delta_2$,
    \begin{eqnarray}
       \cl\DD_{j-\delta_1,k+\delta_2;l+\delta_1+\delta_2}
       & \subseteq &
       \cl\DD_{j,k;l}
       \ \subseteq \
       \cl\DD_{j+\delta_1,k-\delta_2;l-\delta_1-\delta_2},
       \label{eq:delta1}
       \\
       \cl\DD_{j+\delta_1,k-\delta_2;l}
       & \subseteq &
       \cl\DD_{j,k;l}
       \ \subseteq \
       \cl\DD_{j-\delta_1,k+\delta_2;l}.
       \label{eq:delta2}
     \end{eqnarray}
     
\end{proposition}

\begin{proof}
  We first show that $\DD_{j',k';l}\subset\cl\DD_{j,k;1}$ for all
  $l$ with $1\le l\le k'-j'-1$. The first set consists of all points
  $z\in\CC$ such that $\sum_{i=1}^{k'-j'}\alpha(\pm
  a_{j',k';i};z)=l\pi$ in the centro-symmetric coordinates. Undoing
  the coordinate change~\eqref{eq:coo-change} yields
  \begin{eqnarray*}
    \DD_{j',k';l}
    &=&
    \{z\in\CC: \alpha\big({\pm(d-k'+1)};z\big)+\dots+
    \alpha\big({\pm(d-j')};z\big) 
    \ = \ 
    l\pi\},
    \\
    \cl\DD_{j,k;1}
    &=&
    \{z\in\CC: \alpha\big({\pm(d-k+1)};z\big)+\dots+
    \alpha\big({\pm(d-j)};z\big) 
    \ \ \; \ge \ 
    \pi\}.
  \end{eqnarray*}
  Now the required inclusion is clear, because the first set of points
  of which the viewing angle is taken is a subset of the second
  one. It remains to prove that
  $\DD_{j,k;k-j-1}\subset\cl\DD_{j',k';l}$ for all $1\le l\le
  k'-j'-1$; proving the extremal case $l=k'-j'-1$ is sufficient. Thus,
  we are required to show that $\sum_{m=d-k+1}^{d-j}\alpha(\pm
  m;z)=(k-j-1)\pi$ implies $\sum_{m=d-k'+1}^{d-j'}\alpha(\pm
  m;z)\ge(k'-j'-1)\pi$.  But this is true because the first sum has
  $k-j$ summands, the second $k'-j'$ summands, and removing each of
  the $(k-j)-(k'-j')=(k-k')+(j'-j)\ge0$ pairs of points from the
  points corresponding to the first summand decreases the total
  viewing angle by at most~$\pi$.

  Similarly, the first inclusion of \eqref{eq:delta1} follows because
  $\sum_{m=d-k+1-\delta_2}^{d-j+1+\delta_1}\alpha(\pm m;z) \ge
  (l+\delta_1+\delta_2)\pi$ implies 
  $\sum_{m=d-k+1}^{d-j}\alpha(\pm m;z) \ge l\pi$, by removing
  $\delta_1+\delta_2$ pairs of points, and the second one from
  an appropriate change of variables. Relations~\eqref{eq:delta2}
  are proved in exactly the same way. 
  %
\end{proof}

\begin{corollary}\label{cor:cone}
  For $z\notin\cl\DD_{0,d;1}$, the facets of the cone
  $\tau(z)=\pos{w_0(z),\dots,w_d(z)}$ are the rays spanned by $w_0(z)$
  and $w_d(z)$, and $w_1(z),\dots,w_{d-1}(z)$ appear in cyclic order
  inside~$\tau(z)$.
\end{corollary}

\begin{proof}  
  The argument of $w_j(z)=\binom{z+d-j}{d}$ is $\big(\sum_{i=1}^d
  \arg{(z+i-j)}\big)\bmod 2\pi$, so that the difference of the
  arguments of $w_j(z)$ and $w_{j+1}(z)$ equals
  $\beta_j:=\arg(z+d-j)-\arg(z-j)\bmod 2\pi$;
  cf.~\cite{Braun-Develin06}. If we choose $z$ to have the form
  $z=N+i\e$, with $N\gg\e>0$,  it is not necessary to reduce
  $\beta_j$ modulo $2\pi$, and $0<\beta_0<\beta_1<\dots<\beta_d$; we may
  even achieve $\beta_d<\frac{\pi}{d}$, so that the total angle
  subtended by the $w_i(z)$ is strictly less than $\pi$, and $w_0(z)$
  and $w_d(z)$ span the facets of~$\tau(z)$. Now note that
  two vectors $w_i(z), w_j(z)$ become collinear iff there is a (not
  necessarily positive or negative) circuit involving the two, iff
  there is such a cocircuit involving $\ow_i(z),\ow_j(z)$, iff
  $\DD_{j,k}(z)=0$. An invocation of Proposition~\ref{prop:bounded}
  finishes the proof.
\end{proof}

We close with a lemma regarding the relative orientations of $w_j,w_k$
on $\DD_{j,k}$.

\begin{lemma}\label{lem:orientation}
  Let $z\in\DD_{j,k;l}$, and regard $w_i(z)=\binom{z+d-i}{d}$ as a
  vector in $\RR^2$. Then $w_j(z)$ and $w_k(z)$ point in the same
  direction iff $l$~is~even, and in opposite directions iff
  $l$~is~odd:
  \[
     \sign \big(w_j(z)\cdot w_k(z)\big) \ = \ (-1)^{l} \qquad
     \text{for } z\in\DD_{j,k;l}.
  \]
\end{lemma}

\subsection{Conclusion: the root locus is bounded}
\label{sec:bounded}

\begin{theorem}\label{thm:binomialzeros}
  Let $f=\sum_{j=0}^d a_j\binom{z+d-j}{d}$ be a polynomial of degree
  $d$ with nonnegative coefficients $a_j\ge0$ with respect to the
  binomial coefficient basis. Then all non-real roots of $f$ are
  contained in the region $\cl\DD_{0,d;1}$ bounded by the outermost
  oval of the algebraic curve with equation $D_{0,d}(z)=0$, and any
  point inside $\cl\DD_{0,d;1}$ arises as a root of some such $f$.

  The real roots of $f$ all lie in the real interval $[-d,d-1]$.
\end{theorem}

\begin{proof}
  The first statement can be proved by a short calculation involving
  Lemma~\ref{lem:leading} and the general tool of
  Theorem~\ref{thm:roots1}.
%
%
  However, we have accumulated enough information about the special
  curves~$\DD_{j,k}$ arising for the binomial coefficient basis to
  give a direct proof: By Corollary~\ref{cor:cone}, the vectors
  $w_0=w_0(z),\dots,w_d=w_d(z)$ are positively spanning for
  $z\notin\cl\DD_{0,d;1}$.

  Next, suppose that $z\in\cl\DD_{0,d;1}$ falls inside the region
  $S_k:=\cl\DD_{0,k;1}\smallsetminus\cl\DD_{0,k-1;1}$ for some
  $k\in\NN$ with $2\le k\le d$. Such a $k$~exists, because
  $\cl\DD_{0,k-1;1}\subset\cl\DD_{0,k;1}$ by~\eqref{eq:delta2}, and
  $\cl\DD_{0,1;1}=\emptyset$. We claim that in this situation, the
  vectors $w_0$, $w_{k-1}$~and~$w_k$ are positively spanning. Indeed,
  the locus of points in the complex plane where the combinatorics of
  this subconfiguration changes is exactly $\DD_{0,k-1}\cup\DD_{0,k}$,
  because $\DD_{k,k-1}=\emptyset$. Moreover,
  $\DD_{0,k;l}\subset\cl\DD_{0,k-1;1}$ for~$l\ge 2$
  by~\eqref{eq:delta1}, so the boundary of the region~$S_k$ is
  $\DD_{0,k-1;1}\cup\DD_{0,k;1}$, and the property of the three
  vectors being spanning or not remains
  constant inside~$S_k$. Since outside of~$\cl\DD_{0,k;1}$, these
  vectors are \emph{not} positively spanning by Corollary~\ref{cor:cone},
  but this changes when crossing~$\partial S_k$, the second statement
  follows. 

Finally, the case of real roots was dealt with at the
  beginning of the present Section~\ref{sec:binomial}.
\end{proof}

\begin{example}\label{ex:n2}
  Let $d=3$. Then 
  \[
     \oW \ = \ 
     \begin{pmatrix}
       p_0 & 0\\
       -q_0 & p_1\\
       r_0 & -q_1\\
       0 & r_1
     \end{pmatrix},
  \]
  $q_0=2(x+1)^2+2y^2-4$, $q_1=2x^2+2y^2-4$, $p_0=x^2+y^2$,
  $p_1=(x-1)^2+y^2$, $r_0=(x+2)^2+y^2$ and
  $r_1=(x+1)^2+y^2$. Furthermore, $D_{0,2}=-q_0r_1$,
  $D_{0,3}=q_0q_1-p_1r_0$, $D_{1,3}=-p_0q_1$, and $D_{i,j}\ge0$
  otherwise. Now
  \begin{align*}
    \cS_{012} &= \big\{ z:D_{1,2}\ge0, \ D_{0,2}, \ D_{0,1}\ge0\big\}, &
    \cS_{013} &= \big\{ z:D_{1,3}, \ D_{0,3}, \ -D_{0,1}\le 0\big\},\\
    \cS_{023} &= \big\{ z:D_{2,3}\ge0, \ -D_{0,3}, \ -D_{0,2}\big\}, &
    \cS_{123} &= \big\{ z:-D_{2,3}\le0, \ -D_{1,3}, \ -D_{1,2}\le0\big\},
  \end{align*}
  so that by Figure \ref{fig:pic3} and Theorem \ref{thm:binomialzeros}
  all non-real roots of polynomials of degree $3$ with nonnegative
  coefficients in the binomial coefficient basis lie in the union of
  these regions.
\end{example}

\begin{figure}[htbp]
  \centering
  \includegraphics[width=.8\linewidth]{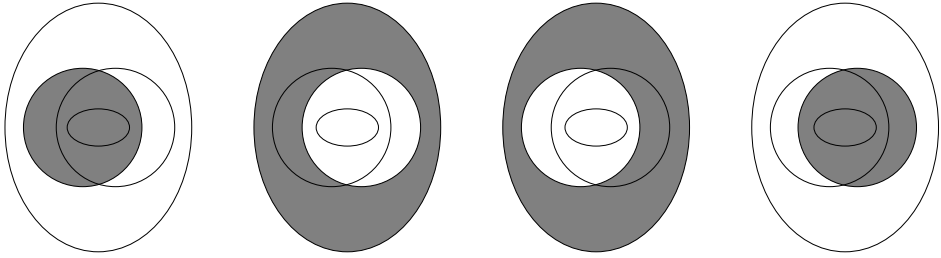}
  \caption{From left to right, the semialgebraic sets $\cS_{012}$,
    $\cS_{013}$, $\cS_{023}$, $\cS_{123}$ (shaded). Their union equals the
    entire interior of the bounding curve $C$, which by
    Theorem~\ref{thm:binomialzeros} is precisely the locus of possible
    non-real roots.}
  \label{fig:pic3}
\end{figure}




\section{Incorporating additional linear constraints}\label{sec:linear}

\subsection{Linear inequalities}

Suppose we not only know that the coefficients $a_i$ of a polynomial
$f=\sum_{i=0}^d a_i b_i$ with respect to some basis $B=\{b_i:i=0,\dots,d\}$ are
nonnegative, but also that they satisfy a linear inequality
$\sum_{i=0}^d\lambda_i a_i\le0$; the `$\ge0$'~case is of course 
accounted for by reversing the signs of the~$\lambda_i$. We use a
slack variable $s\ge0$ to rewrite our inequality as
\[
  \sum_{i=0}^d \lambda_i a_i + s \ = \ 0.
\]
To incorporate this into our Gale dual matrices $W$ and
$\oW$, we introduce the vector $\ta=(a_0,\dots,a_d,s)^T$. The analogue
$W\ta=0$ of~\eqref{eq:primal} is
\[
  \begin{pmatrix}
    R_0 & R_1 & \dots & R_d & 0\\
    I_0 & I_1 & \dots & I_d & 0\\
    \lambda_0 & \lambda_1 & \dots & \lambda_d & 1
  \end{pmatrix}
  \begin{pmatrix}
    a_0\\ a_1\\\vdots\\a_d\\ s
  \end{pmatrix}
  \ = \ 0,
\] 
and we name the columns of this new $W$ by $w_0,\dots, w_{d+1}$.
We obtain a Gale dual $\tW=\tW(z)$ of $W$ by appending
the row vector
\[
   \ow_{d+1} \ = \ 
   (\omega_0, \dots, \omega_{d-2}) \ = \
   \big({-\lambda_i p_i} + \lambda_{i+1} q_i - \lambda_{i+2} r_i
   \ : \ 0\le i\le
   d-2\big)
\]
to the matrix $\oW$ from \eqref{eq:dualmatrix}. For the polynomial
$f(z)=\sum_{i=0}^d a_i b_i(z)$ with $a_i\ge0$ and $\sum_{i=0}^d
\lambda_i a_i\le 0$ to have a zero at $z=z_0$, the vector~$\ta$ must
lie in the column space of~$\tW(z_0)$ and have nonnegative entries;
equivalently, there must exist a vector
$\mu=(\mu_0,\dots,\mu_{d-2})^T$ with $\tW(z_0)\mu=\ta$. Geometrically,
we think of $\mu$ as the normal vector of a linear hyperplane that
leaves all vectors $\ow_i$ (weakly) on one side. In particular, if the
linear inequality is strict (so that $s>0$), then we are  only
interested in linear hyperplanes that do not contain $\ow_{d+1}$.

In general, $m\ge1$ independent linear inequalities yield a
$(d+1+m)\times(d-1)$-matrix $\tW$.  Consider the configuration of
$d+1+m$ vectors in $\RR^{d-1}$ spanned by the rows of $\tW$. Each
$(d-2)$-tuple of vectors among these spans a linear hyperplane, and we
would like to know when the $m+3$~remaining vectors all lie on the
same side of it.  As before, we treat strict inequalities by only
considering those linear hyperplanes that do not contain any of the
$m$~``new'' vectors~$\ow_j$, and to simplify the discussion we will
focus on these.

We thus fix an ordered subset $J=\{j_1,\dots,j_{m+3}\} =
\{j_1< j_2 < j_3\}\cup\{d+1,\dots,d+m\}$ of $\{0,\dots,d+m\}$; this
set will index the rows of $\tW$ \emph{not} on a linear hyperplane.
Next, we calculate a linear form $\varphi_\bJ$ whose vanishing locus
is the hyperplane spanned by the $d-2$ vectors not indexed by $J$: it
is the determinant of the matrix obtained by deleting from $\tW$ all
rows indexed by $J$, and adding a first row of variables. The sign
$\sigma_{\bJ,i}(z)$ of $\varphi_\bJ(w_{j_i})$ at a point $z\in\CC$ is
then obtained by plugging the coordinates of $w_{j_i}=w_{j_i}(z)$ into
these variables, i.e., by \emph{not} deleting the row with index
$j_i$, but instead permuting it to the first row and then taking the
sign of the determinant of the resulting matrix.  More precisely, if
we denote by~$\tW_K$ the matrix obtained from $\tW$ by deleting the
rows indexed by $K\subset\{0,\dots,d+m\}$, then
\begin{equation}\label{eq:sigma}
  \sigma_{\bJ,i}(z) \ = \ 
  (-1)^{j_i+i+1}\sign\det\tW_{J\smallsetminus\{j_i\}}(z),
  \qquad\text{for }i=1,\dots,m+3.
\end{equation}
                               
Writing
$\sigma(\bJ,z)=\{\sigma_{\bJ,1}(z),\dots,\sigma_{\bJ,m+3}(z)\}$, we
can summarize our discussion as follows:

\begin{theorem}\label{thm:ineqs}
  Assume that the coefficients of $f$ satisfy $m\ge1$ strict linear
  inequalities, indexed from $d+1$ to $d+m$.  Let
  \[
     \cS(J) \ = \ 
     \big\{z\in\CC: \sigma(\bJ,z)=\{-1,0\} \text{ or }
     \sigma(\bJ,z)=\{0,+1\}\big\}. 
  \]
  Then the set of roots of $f$ is contained in the union $\bigcup_{J}
  \cS(J)$, where $J$ runs through all sets of the form $\{j_1, j_2,
  j_3\}\cup\{d+1,\dots,d+m\}$ with $0\le j_1<j_2<j_3\le d$; put
  differently, if $\{-1,1\}\subseteq\sigma(\bJ,z)$ for each such $J$,
  then $z_0$ is not a root of $f$.\hfill$\Box$
\end{theorem}

In the case $m=1$ and $J=\{j,k,l,d+1\}$, we obtain from
\eqref{eq:sigma} that 
\begin{equation}\label{eq:sigmaJ}
    \sigma(\bJ,z) \ = \ \bigg\{ (-1)^{j} \sign D_{k,l},\
    (-1)^{k+1} \sign D_{j,l}, \
    (-1)^{l} \sign D_{j,k},\
    (-1)^d \sign\det\tW_{\{j,k,l\}}
    \bigg\}.
\end{equation}
Expanding the last determinant along its last row yields
\begin{equation}\label{eq:thedet}
  \det \tW_{\{j,k,l\}} \ = \
  (-1)^d \sum_{c=0}^{d-2} (-1)^c \omega_i [\oW]_{\{j,k,l\};c},
\end{equation}
where $[\oW]_{\{j,k,l\};c}$ stands for the minor of~$\oW$ obtained by
deleting rows $j,k,l$ and column $c$. This formula can be evaluated as
follows:

\begin{lemma}\label{lem:minor}
  Let $m=1$, $0\le j<k<l\le d$, and $0\le c\le d-2$.  Then
  \[
  [\oW]_{\{j,k,l\};c} \ = \
  \begin{cases}
    p_0\cdots p_{j-1} D_{j,c+1} p_{c+1}\cdots p_{k-1} D_{k,l} 
    r_{l-1}\cdots r_{d-2} & \text{if }\; 0\le c \le k-1,\\
    p_0\cdots p_{j-1} D_{j,k} r_{k-1}\cdots r_{c-1} D_{c+1,l}
    r_{l-1}\cdots r_{d-2} & \text{if }\; k-1\le c\le d-2.\\
  \end{cases}
  \]
  Here we follow the convention that $p_a\cdots p_b=r_a\cdots r_b=1$
  if $a>b$, but $D_{a,b}=0$ for $a\ge b$. In particular,
  $[\oW]_{\{j,k,l\};c}=0$ for $0\le c\le j-1$ and $l-1\le c\le d-2$. 
\end{lemma}

\begin{proof}
  In each case, $\oW_{K;c}$ decomposes into square blocks on the
  diagonal whose determinants yield the stated expressions. The
  elements outside these blocks do not contribute to $[\oW]_{K;c}$,
  because the determinant of a block matrix of the form $
  (\begin{smallmatrix} A & 0 \\ C & D
  \end{smallmatrix})$ or $
  (\begin{smallmatrix}
    A & B \\ 0 & D
  \end{smallmatrix})$ is $\det A\det D$.
\end{proof}

To recapitulate,  additional linear inequalities can only restrict
further the location of possible roots of~$f$. If the vectors
$w_i(z_0),w_j(z_0),w_k(z_0)\in\cB$ do not witness a possible root
of~$f$, in other words $\{\pm1\}\subseteq \big\{ (-1)^{j} \sign
D_{k,l}(z_0),\ (-1)^{k+1} \sign D_{j,l}(z_0), (-1)^{l} \sign
D_{j,k}(z_0)\big\}$, nothing changes after incorporating the
additional sign $(-1)^d \sign\det\tW_{\{j,k,l\}}(z_0)$: the vectors
$w_i(z_0),w_j(z_0),w_k(z_0),w_l(z_0)$ do still not witness a root
of~$f$ at~$z_0$. If, on the other hand, the new sign is different from
the old ones,  there is ``one reason less'' for
$z_0$~to be a root.

\subsection{Linear equations}  

If the coefficients of $f$ satisfy $m$~independent linear \emph{equations} of
the form $\sum_{i=0}^d\lambda_i a_i = 0$ (corresponding to the case
$s=0$), the $d-1-m$~columns of the new Gale dual $\tW$ will of course
be linear combinations of the columns of the old one, but in general
we will not be able to give an explicit expression for them.  We
therefore only treat some special cases that arise in the context of
Ehrhart and chromatic polynomials, and defer further discussion to
Section~\ref{subsec:chromatic}.

\section{Applications}
\label{sec:applications}

\subsection{Ehrhart polynomials}
\label{sec:ehrhart}

From \cite{Beck-etal05}, we  know that the following inequalities
hold for the coefficients of $i_Q$ in the binomial basis:
\[
  a_d + a_{d-1} + \dots + a_{d-s} \ \le \ 
  a_0 + \dots + a_s + a_{s+1}
  \qquad\text{for all }
  0\le s \le \fl{(d-1)/2}.
\]

For $s=0$, the inequality reads $a_d\le a_0+a_1$, and $\ow_{d+2}$ is
\[
   \big( p_0 - q_0,\, p_1,\, 0,\,\dots,\, 0,\, -r_{d-2}\big).
\]
Equation \eqref{eq:thedet} and Lemma \ref{lem:minor} thus specialize as follows:
\begin{eqnarray*}
  \det \tW_{\{0,k,d\}} & = &
  \begin{cases}
    (-1)^d(p_0-q_0)D_{1,d} - (-1)^d p_1r_0 D_{2,d} - r_0\cdots r_{d-2}
    & \text{if } k=1,\\
    (-1)^dp_0\cdots p_{k-1} D_{k,d} - D_{0,k} r_{k-1}\cdots r_{d-2}
    & \text{if } 2\le k\le d-1,
  \end{cases}\\
  \det\tW_{\{1,k,l\}} &=&
  (-1)^d q_0p_0\cdots p_{k-1} D_{k,l} r_{l-1}\cdots r_{d-2}
  - p_0 D_{1,k}r_{k-1}\cdots r_{d-2} \iverson{l=d}
\end{eqnarray*}
(Here we have used Iverson's notation: $\iverson{l=d}$ evaluates to
$1$ if $l=d$, and to $0$ otherwise.)  Explicit calculation using
Lemma~\ref{lem:leading} yields that the coefficient of the leading
term $r^{2d-2}$ in $(-1)^d \det\tW_{\{0,k,d\}}$ is $2(-1)^{d+1}$ for $k=1$ and
$(-1)^{d+k+1}(d-2k)$ for $2\le k\le d-1$. Thus, the sign of this
coefficient is
\begin{equation}\label{eq:signcoeff}
  \sign \big([r^{2d-2}] (-1)^d \det\tW_{\{0,k,d\}}\big) \ = \
  \begin{cases}
    (-1)^{d+k+1} &\text{for } 0< k<d/2, \\
    (-1)^{d+k} & \text{for } d/2< k < d.
  \end{cases}
\end{equation}

We examine the effect that this has on $\sigma(\bJ,z)$. If
$z\in\CC$ does not lie in $\cl D_{0,d;1}$, the first three entries
of~\eqref{eq:sigmaJ} already yield two different signs, no matter what
sign the last determinant takes. Now let $z$ lie inside $\cl
D_{0,d;1}$, but outside the union of all $\cl D_{i,j;1}$ with
$(i,j)\ne(0,d)$. If $\{j,k,l\}$ does not contain $\{0,d\}$, the first
three signs of $\sigma(\bJ,z)$ in~\eqref{eq:sigmaJ} will again contain
two different ones. The interesting situation is thus $J=\{0,k,d\}$,
in which case $\sigma(\bJ,z)=\big\{(-1)^{k+d+1},
(-1)^d\sign\det\tW_{\{0,k,d\}}(z)\big\}$. Combining this
with~\eqref{eq:signcoeff}, we see that these signs are different,
i.e., $z$'s ``last opportunity'' $J$ also does not make it  an
Ehrhart zero, if $z$~lies inside the outermost component of the zero
locus of $\det\tW_{\{0,k,d\}}$ for $0<k<d/2$, but outside all
components of $\det\tW_{\{0,k,d\}}$ for $d/2<k<d$.
Figure~\ref{fig:plotd10} shows that this actually
occurs.

\begin{figure}[htp]
  \centering
  \includegraphics[height=.25\linewidth]{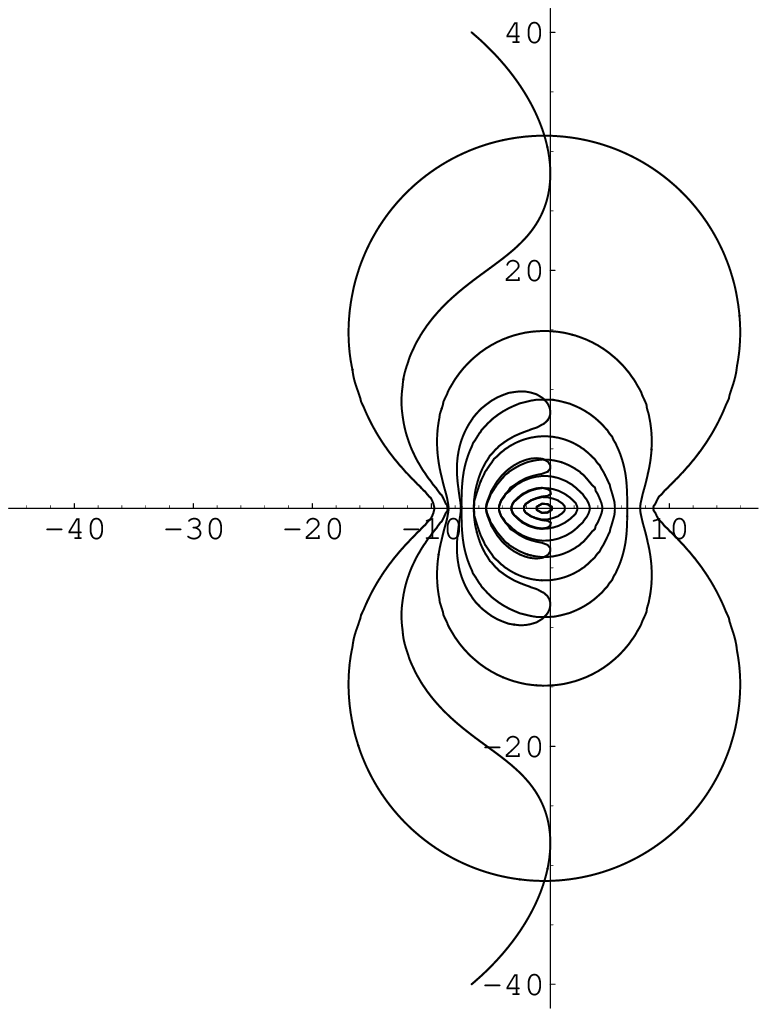}
  \includegraphics[height=.25\linewidth]{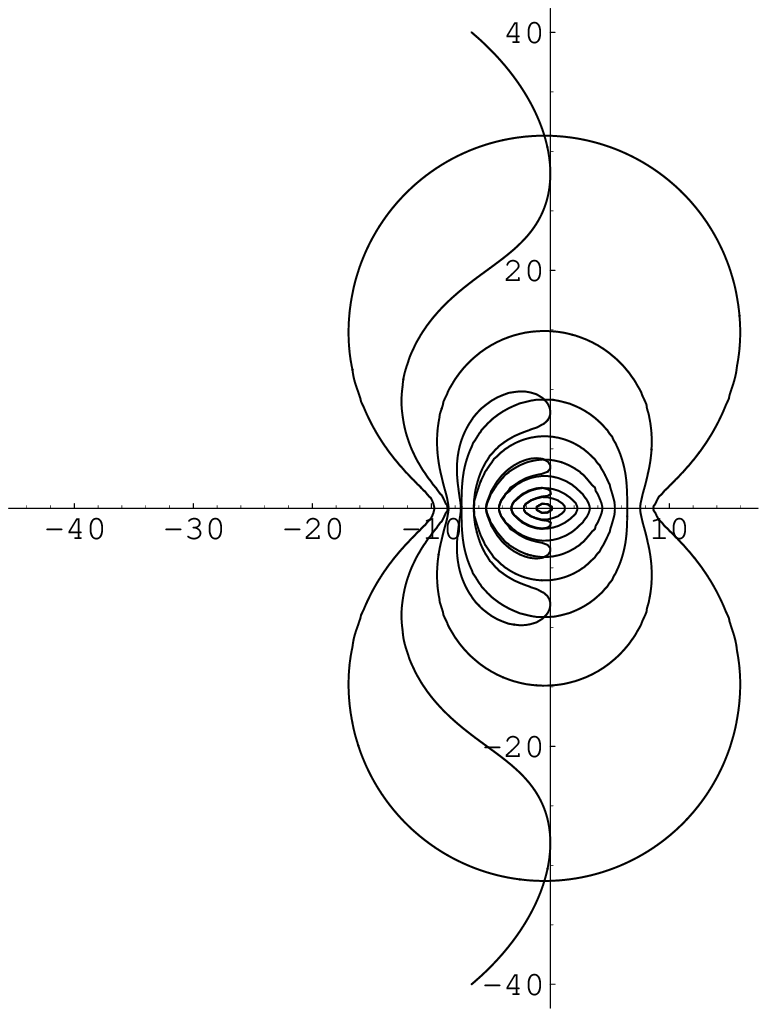}
  \includegraphics[height=.25\linewidth]{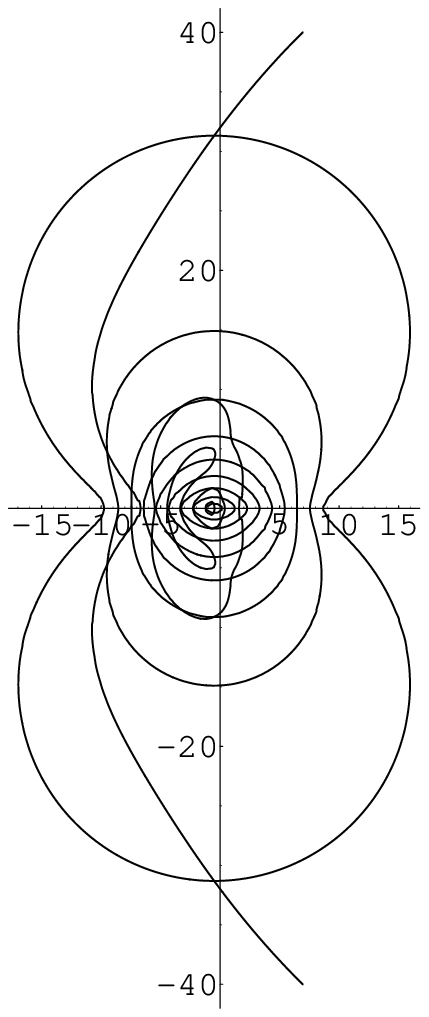}
  \includegraphics[height=.25\linewidth]{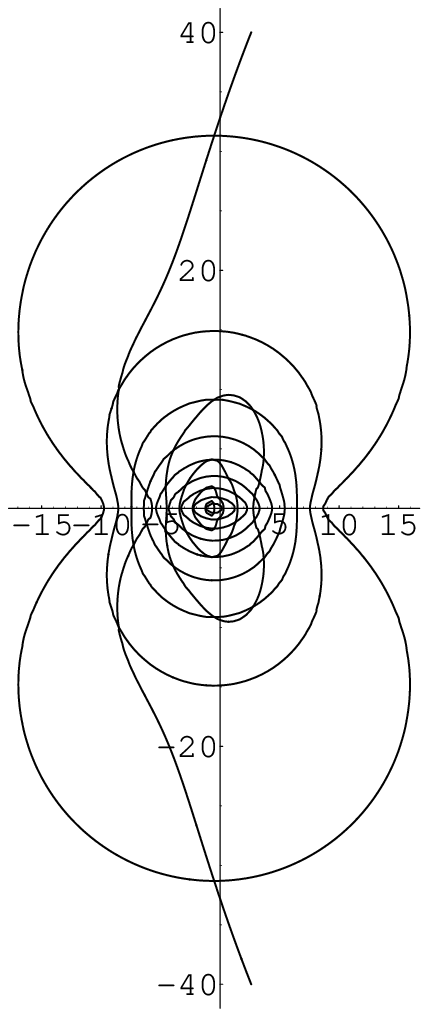}
  \includegraphics[height=.25\linewidth]{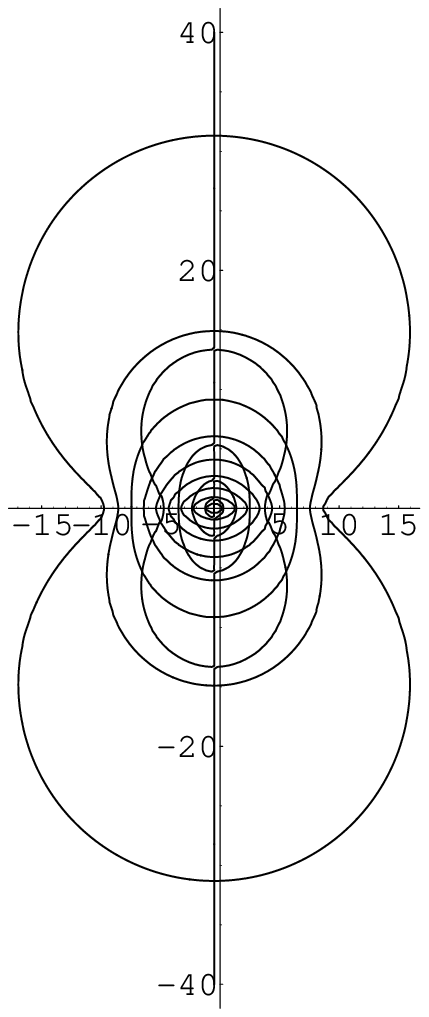}

  \includegraphics[height=.25\linewidth]{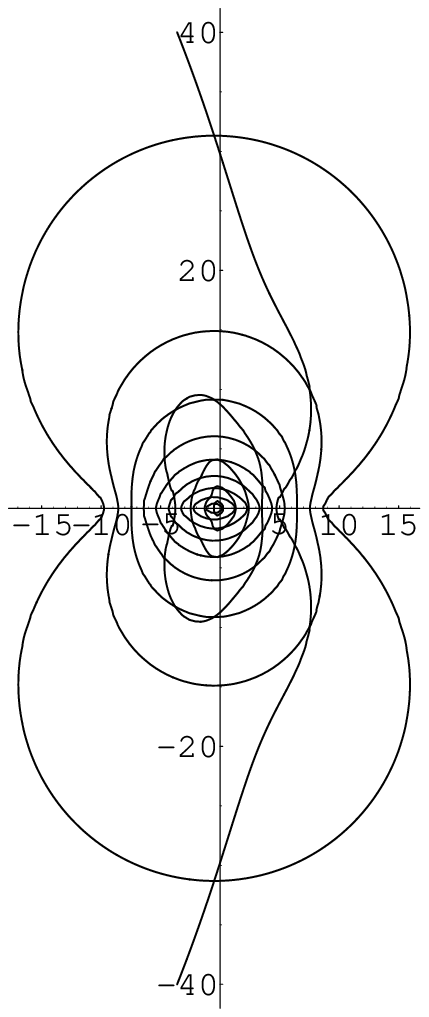}
  \includegraphics[height=.25\linewidth]{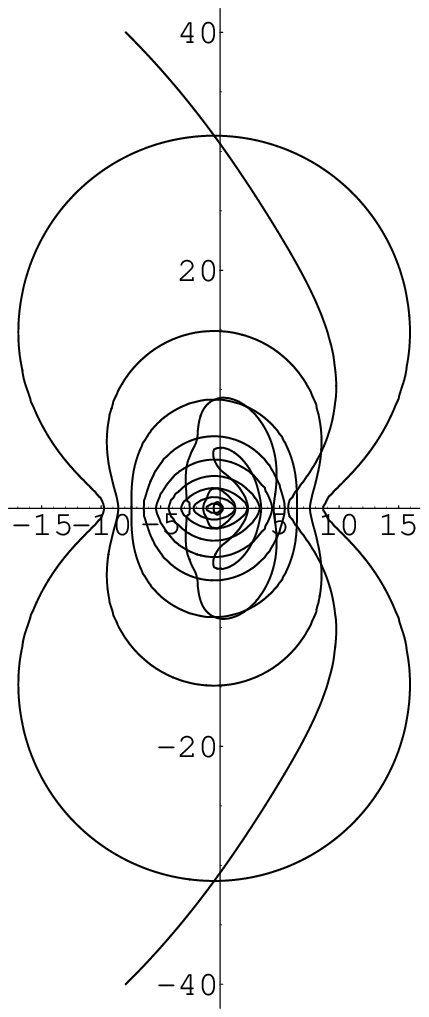}
  \includegraphics[height=.25\linewidth]{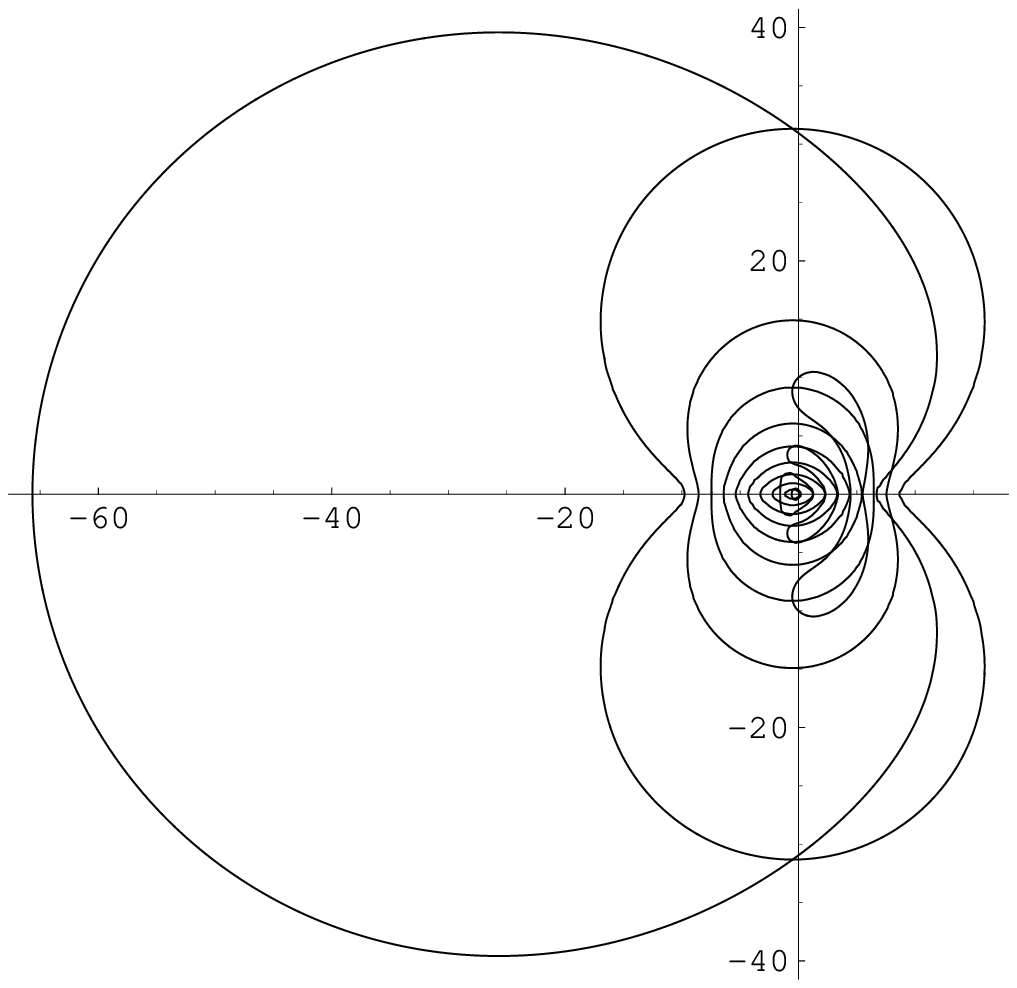}
  \includegraphics[height=.25\linewidth]{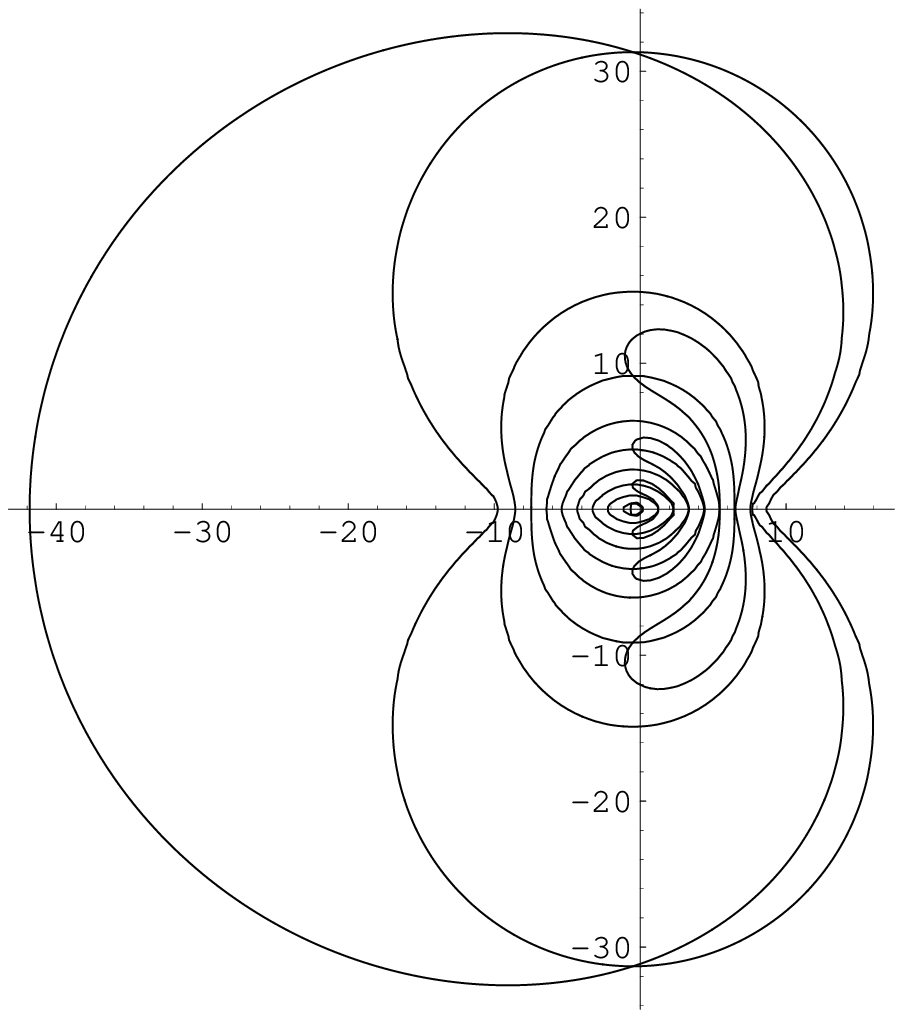}

  \caption{From top left to bottom right, the vanishing loci of
    $D_{0,d}$ and $\tW_{0,k,d}$, for $d=10$, $j=0$, $l=d$, and $1\le
    k\le d-1$. In the first four pictures, the outermost oval is only
    partially shown, but in fact curves around to the right and
    intersects the real axis at a point with large positive
    coordinate. Thus, the points with positive real part just
    inside~$\DD_{0,d;1}$ lie inside the outermost component of the
    zero locus of $\det\tW_{\{0,k,d\}}$ for $0<k<d/2$, but outside all
    components of $\det\tW_{\{0,k,d\}}$ for $d/2<k<d$.}
  \label{fig:plotd10}
\end{figure}


\subsection{Chromatic polynomials}\label{subsec:chromatic}

Let $G$ be a graph on $d$ vertices. The value of the chromatic
polynomial $P(G,t)$ of~$G$ at $z=t_0$ counts the number of colorings
of~$G$ with $z_0$~colors.  The chromatic number $\chi(G)$ is the first
positive integer that is not a zero of $P(G,t)$.

\begin{proposition}\label{prop:chromatic}  
  Let $G$ be an undirected graph on $d$ vertices with $m$ edges,
  $\kappa$~connected components, chromatic number $\chi=\chi(G)$, and
  $\omega$ acyclic orientations. Let $P(G,z)=\sum_{i=0}^d a_ib_i$ be
  the chromatic polynomial of~$G$ expressed in the basis
  $B=\{b_0,\dots,b_d\}$ of $P_d$.
  \begin{enumerate}[\upshape(a)]
  \item Let $b_i=(-1)^{d-i}z^i$, so that $B$ is the \emph{alternating
      power basis}. Then $a_i\ge0$ for $i=0,\dots,d$, $a_i=0$ for
    $i=0,\dots,\kappa-1$, $a_\kappa>0$, $a_{d-1}=m$ and $a_d=1$.
    \cite[Theorem~2.7]{Read-Tutte88}.
    
  \item Let $b_i=z^{\underline i}$, so that $B$ is the \emph{falling
      factorial basis}. Then $a_i\ge0$ for $i=0,\dots,d$
    \cite[Theorem~2.1]{Read-Tutte88}.
    
  \item Let $b_i=(-1)^{d-i}z^{\overline i}$, so that $B$ is the
    \emph{alternating rising factorial basis}. Then $a_i\ge0$ for
    $i=0,\dots,d$ \cite[Proposition~2.1]{Brenti-Royle-Wagner94}.
     
  \item Let $b_i= \binom{z+d-i}{i}$, so that $B$ is the \emph{binomial
      coefficient basis}. Then $a_i\ge0$ for $i=0,\dots,d$,
    $\sum_{i=0}^d a_i = d!$, $a_i=0$ precisely for $0\le i\le \chi-1$,
    and $a_d=\omega$ \cite[Proposition
    4.5]{Brenti92},~\cite{Brenti-Royle-Wagner94}.
  \end{enumerate}
\end{proposition}

The roots of chromatic polynomials  simultaneously satisfy
all restrictions implied by these nonnegativity conditions. Here we
only treat two of these in any detail. 

\subsubsection{The alternating power basis}
To evaluate these conditions for the alternating power basis, only
slight modifications from the power basis case are needed. First,
$q_k=-2x$ instead of $q_k=2x$ in Proposition~\ref{prop:dualmatrix},
and so
\[
    D_{j,k}(z) \ = \ -\frac{z^{n+1}-\bz^{n+1}}{\bz-z}
\]
with $n=k-j-1$.  Next, the relations $a_0=\dots=a_{\kappa-1}=0$ say
that effectively,
\[
  \begin{pmatrix}
    R_\kappa & R_{\kappa+1} & \dots & R_d  \\
    I_\kappa & I_{\kappa+1} & \dots & I_d 
  \end{pmatrix}
  \begin{pmatrix}
    a_\kappa\\ \vdots \\ a_d 
  \end{pmatrix} 
  \ = \ 0,
\]
so that $\oW$ starts out with the column
$(p_\kappa,-q_\kappa,r_\kappa,0,\dots,0)^T$. But in the present case
of the alternating power basis, none of $p_k,q_k,r_k$ actually depends
on~$k$. The matrix~$\oW$ thus stays the same, only the effective
dimension has dropped to $d'=d-\kappa$.  The discussion in
Section~\ref{subsec:power} still applies, except that the excluded
region for roots of~$P(G,z)$ is now the opposite half-open sector,
i.e., the cone $\tau$ bounded by the lines of angles
$\pm(1-\frac{1}{d-\kappa})\pi$.
     
We may incorporate the linear equation $ma_d-a_{d-1}=0$ by appending
the row vector $(0,\dots,0,-1,m)$ of length $d-\kappa+1$ to~$W$, and
replacing the last two columns of $\oW$ by their linear combination
$(0,\dots,0,g,h,m,1)^T$ with $g=(m-2x)(x^2+y^2)$ and
$h=(m-2x)2x+x^2+y^2$. The rows of the resulting matrix $\oW'$
represent $d'+1$ vectors in $\RR^{d'-2}$, so any linear hyperplane
spanned by members of this set is defined by a linear form
$\varphi_{i,j,k,l}$. The signs of the values of this linear form on
the four row vectors $\ow'_i,\ow'_j,\ow'_k,\ow'_l$ are
\begin{eqnarray*}
   \sigma_{i,j,k,l} & = &
   \bigg\{
     (-1)^i \sign\det\oW'_{(j,k,l)},\
     (-1)^{j+1} \sign\det\oW'_{(i,k,l)},\\
     && \phantom{\bigg\{}
     (-1)^k \sign\det\oW'_{(i,j,l)},\
     (-1)^{l+1} \sign\det\oW'_{(i,j,k)}
   \bigg\},
\end{eqnarray*}
where $\oW'_{(i,j,k)}$, for instance, is 
obtained
from $\oW'$ by deleting rows $i,j,k$.  The sets of signs
\[
   \sigma_{i,j,k,d-1} \ = \ \bigg\{(-1)^i \sign D_{j,k},  \,
   (-1)^{j+1} \sign D_{i,k}, \, (-1)^k \sign D_{i,j}, 0 \bigg\}
\]
tell us that any root allowed by the conditions $a_i\ge0$ is also
allowed under the additional restriction $ma_d=a_{d-1}$, so that the
set of possible roots does not change under this restriction. 

\subsubsection{The binomial coefficient basis}

The relations $a_0=\dots=a_{\chi-1}=0$ say that effectively,
\[
  \begin{pmatrix}
    R_\chi & R_{\chi+1} & \dots & R_d  \\
    I_\chi & I_{\chi+1} & \dots & I_d 
  \end{pmatrix}
  \begin{pmatrix}
    a_\chi\\ \vdots \\ a_d 
  \end{pmatrix} 
  \ = \ 0,
\]
so that $\oW$ starts out with the column
$(p_\chi,-q_\chi,r_\chi,0,\dots,0)^t$. The transformation $x\mapsto
x+\chi$ maps $(p_{\chi+i},-q_{\chi+i},r_{\chi+i})$ to
$(p_i,-q_i,r_i)$, so after this translation the effective dimension
has dropped to $d'=d-\chi$. 

The two affine linear relations $a_\chi+\dots+a_d=d!$ and $a_d=\omega$
of course do not individually influence the location of roots, but may
be combined to the linear relation $\sum_{i=\chi}^{d}a_i -
\frac{1}{\e}a_d=0$ with $\e=\frac{\omega}{d!}$. A Gale dual compatible
with this linear relation is the matrix $\tW$ of size
$(d'+1)\times(d'-2)$ with columns
\[
   \tW \ = \ \left( v_1-v_0,\ \dots,\ v_{d'-3}-v_0,\ 
     \lambda v_{d'-2} - \mu v_0 \right),
\] 
where $v_{i}$ is the $i$-th column of $\oW$, and the
coefficients are $\lambda=\e d(d-1)$ and $\mu=\lambda-r_{d'-2}$. To
calculate the sets $\sigma_{i,j,k,l}$ of signs, we must evaluate the
determinant $[\tW]_K$ of the submatrix of $\tW$ obtained by deleting
the three rows indexed by $K=\{i,j,k\}$, say. By multilinearity of the
determinant, we obtain
\begin{eqnarray*}
  [\tW]_K &=& 
  \lambda\det\left( 
    v_1-v_0,\ \dots, v_{d'-3}-v_0,\ v_{d'-2}
  \right) - (-1)^{d'-3} 
  \mu\det\left( v_0,\ v_1,\ \dots, v_{d'-3}\right) \\
  &=&
  \lambda\sum_{c=0}^{d'-3} (-1)^c [\oW]_{K;c} + 
  (-1)^{d'-2} \mu [\oW]_{K;d'-2} \\
  &=&
  \frac{\omega}{(d-2)!}\sum_{c=0}^{d'-2} (-1)^c [\oW]_{K;c} - 
  (-1)^{d'-2} r_{d'-2} [\oW]_{K;d'-2}.
\end{eqnarray*}
This formula can be evaluated using Lemma~\ref{lem:minor}. In
Figure~\ref{fig:ChrEq} we show the zero loci of $[\tW]_K$ in the case
$d=4$ and $\omega=\frac{d!}{2}$.

\begin{figure}[htbp]
  \centering
  \includegraphics[width=.17\linewidth]{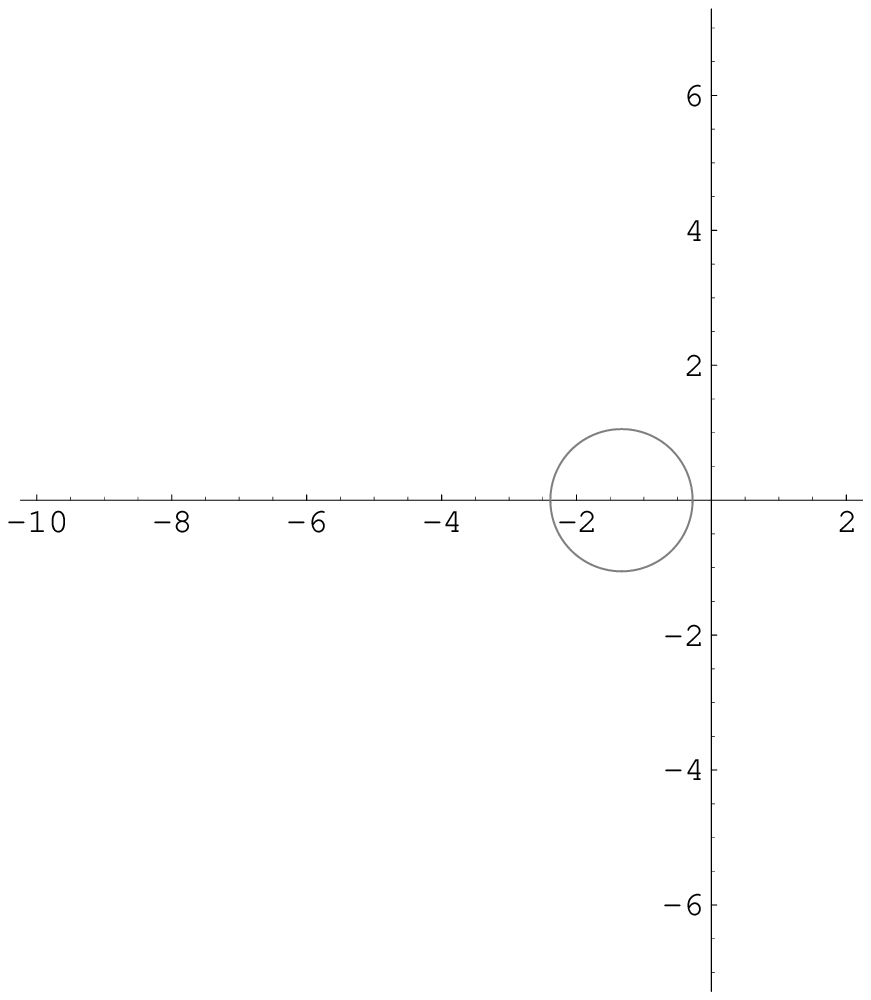}
  \includegraphics[width=.17\linewidth]{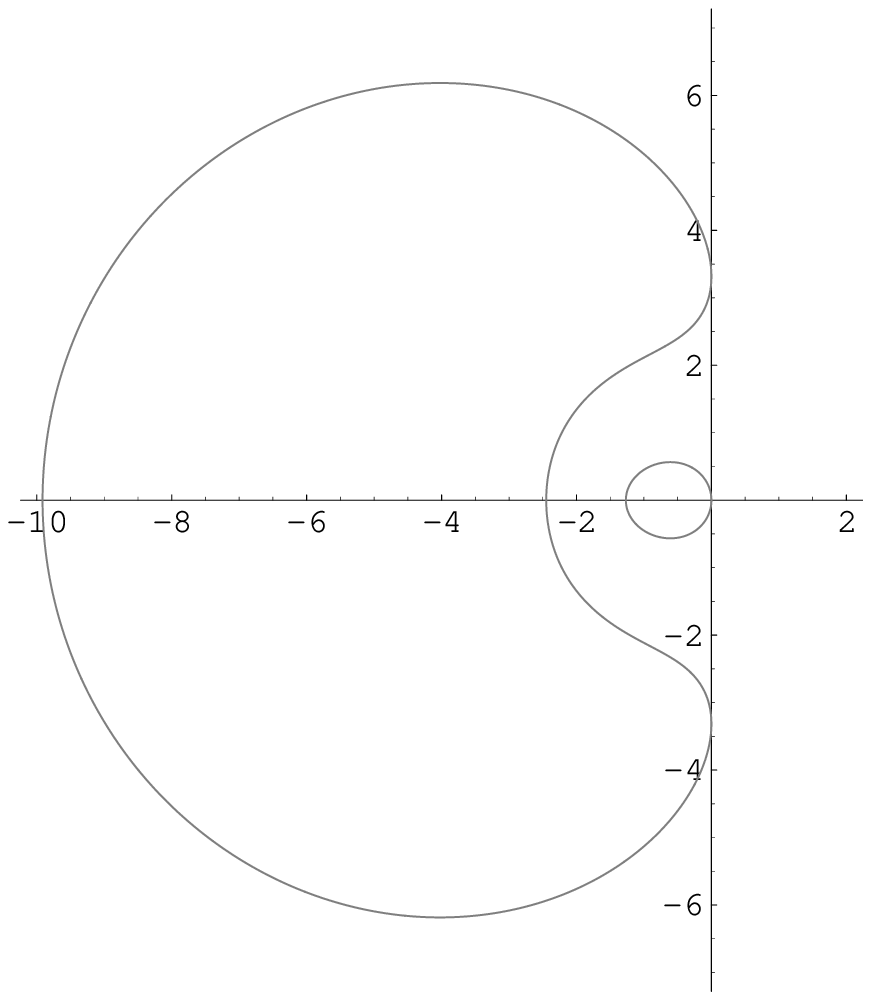}
  \includegraphics[width=.17\linewidth]{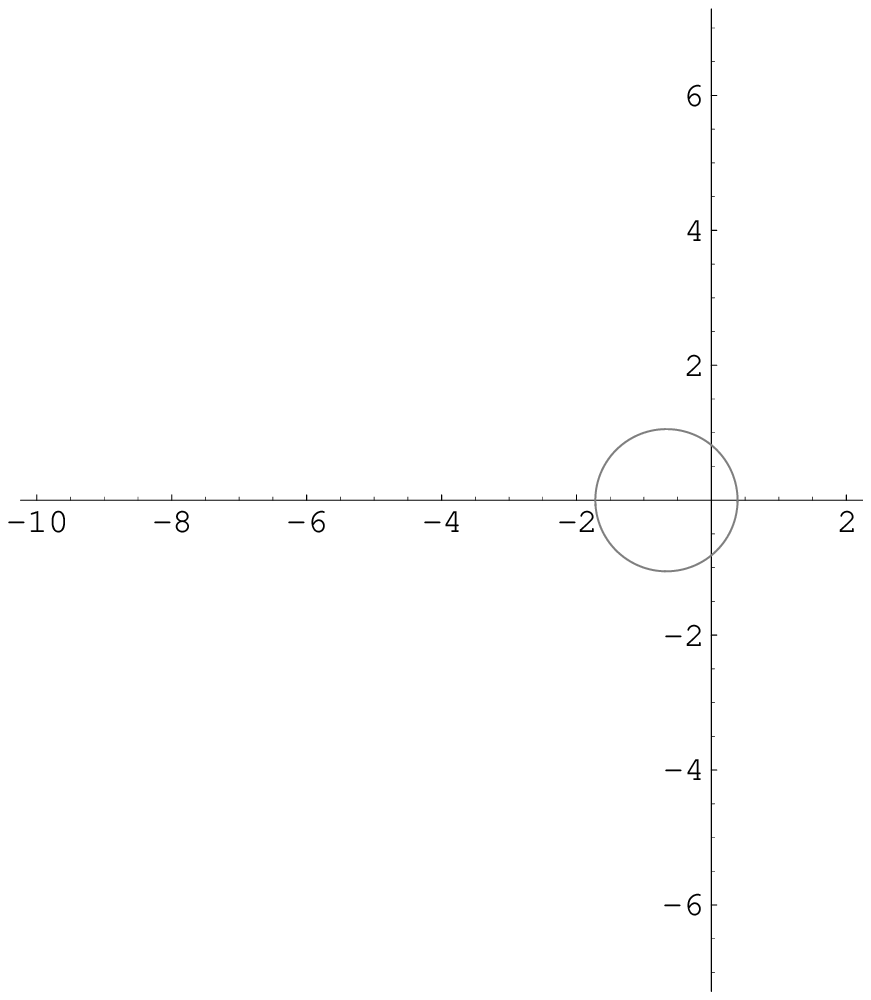}
  \includegraphics[width=.17\linewidth]{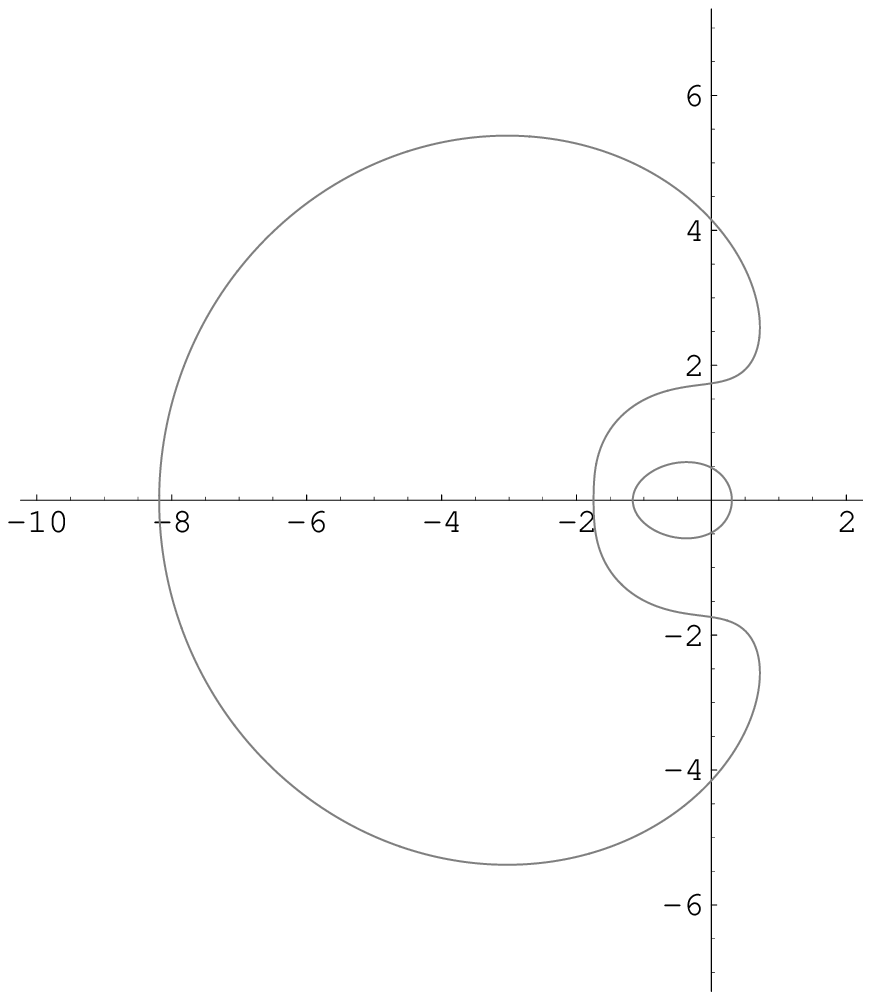}

  \includegraphics[width=.17\linewidth]{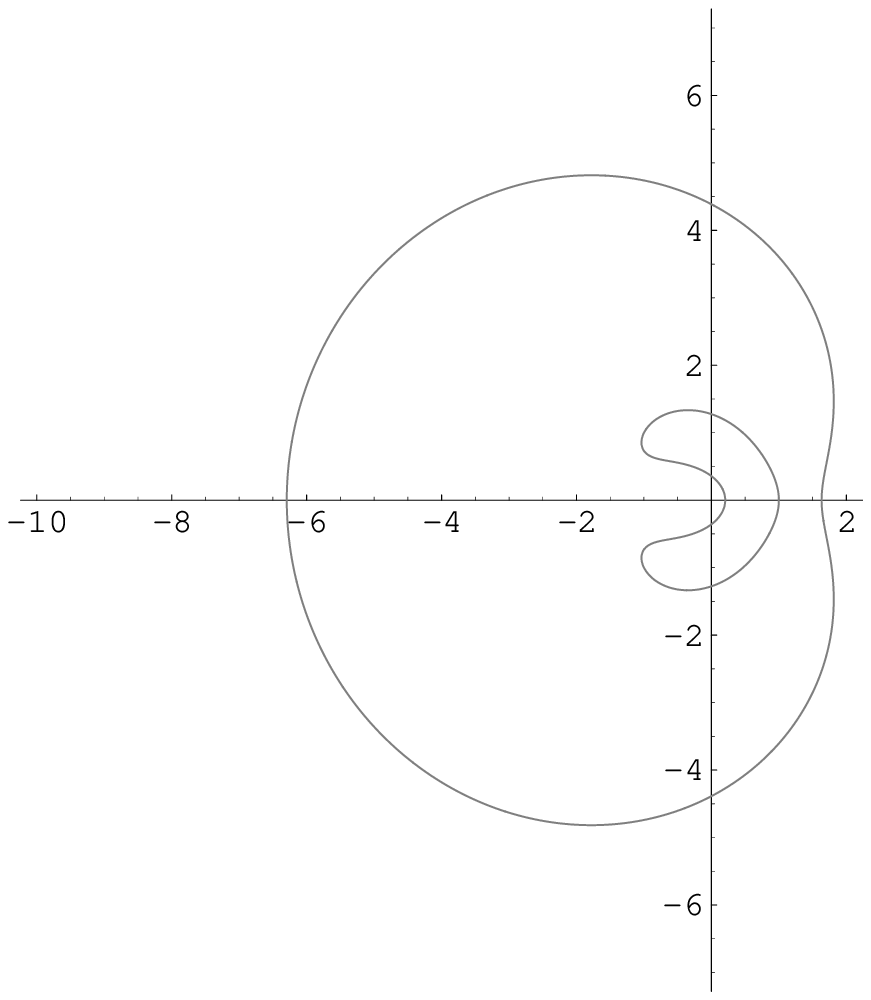}
  \includegraphics[width=.17\linewidth]{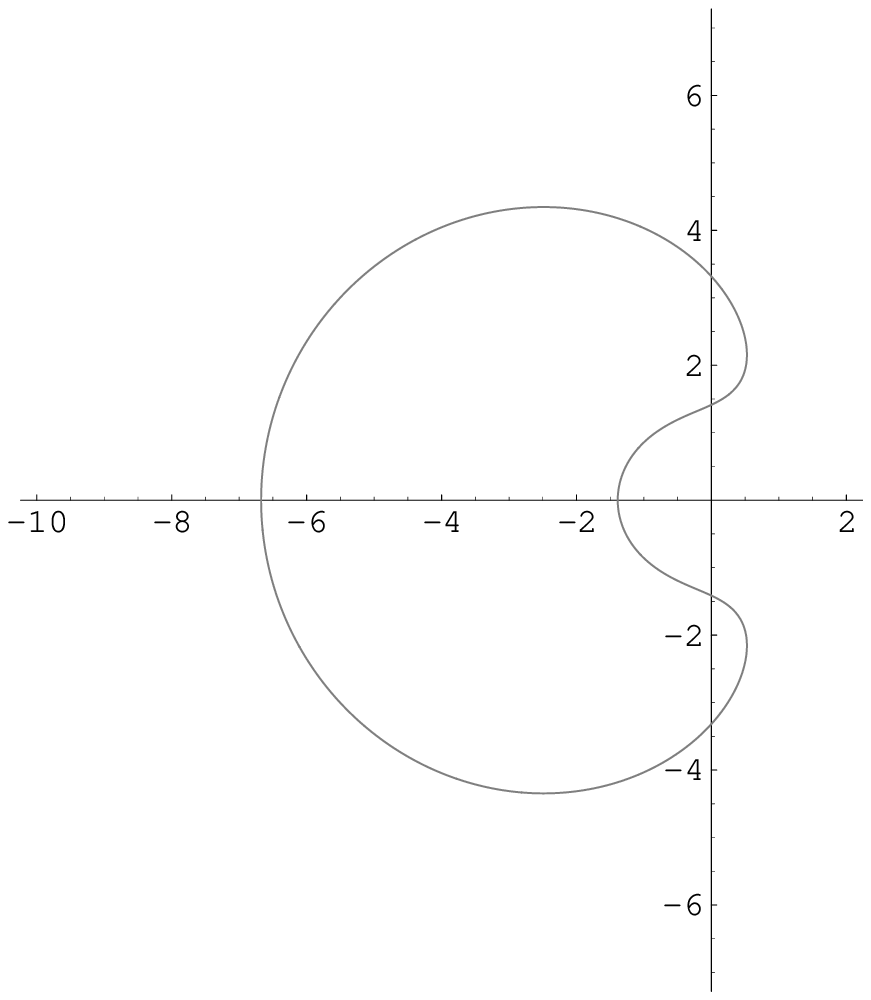}
  \includegraphics[width=.17\linewidth]{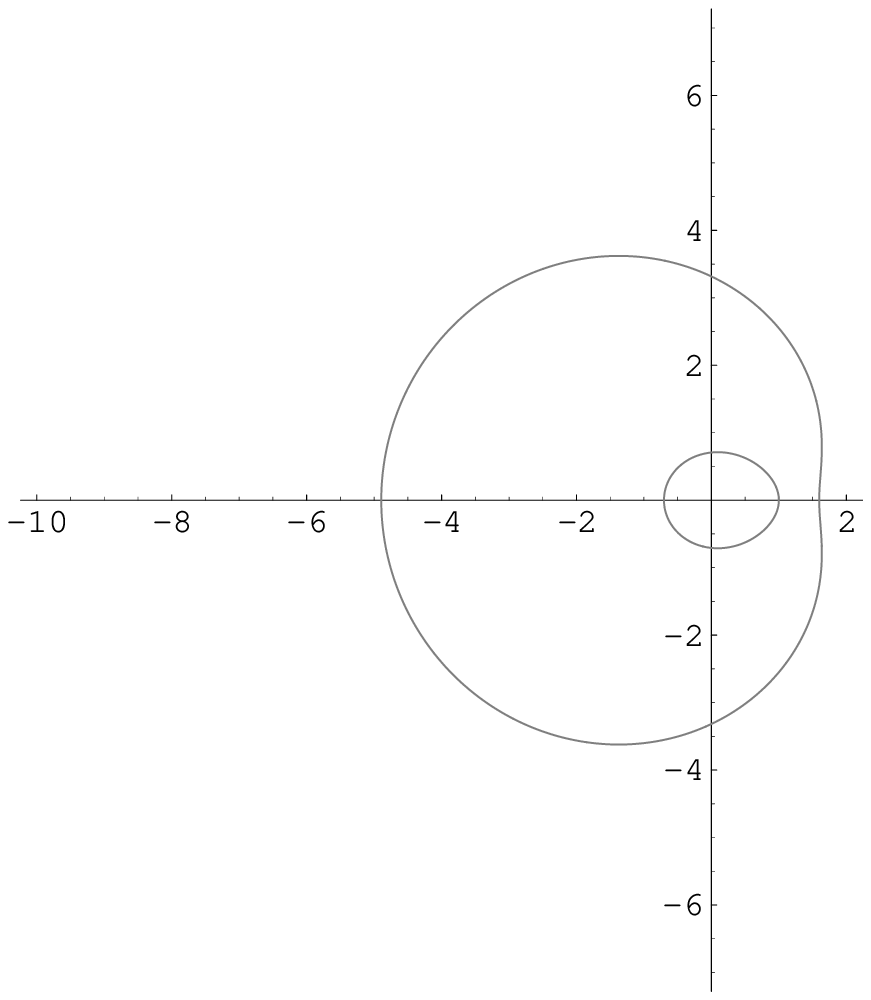}
  \includegraphics[width=.17\linewidth]{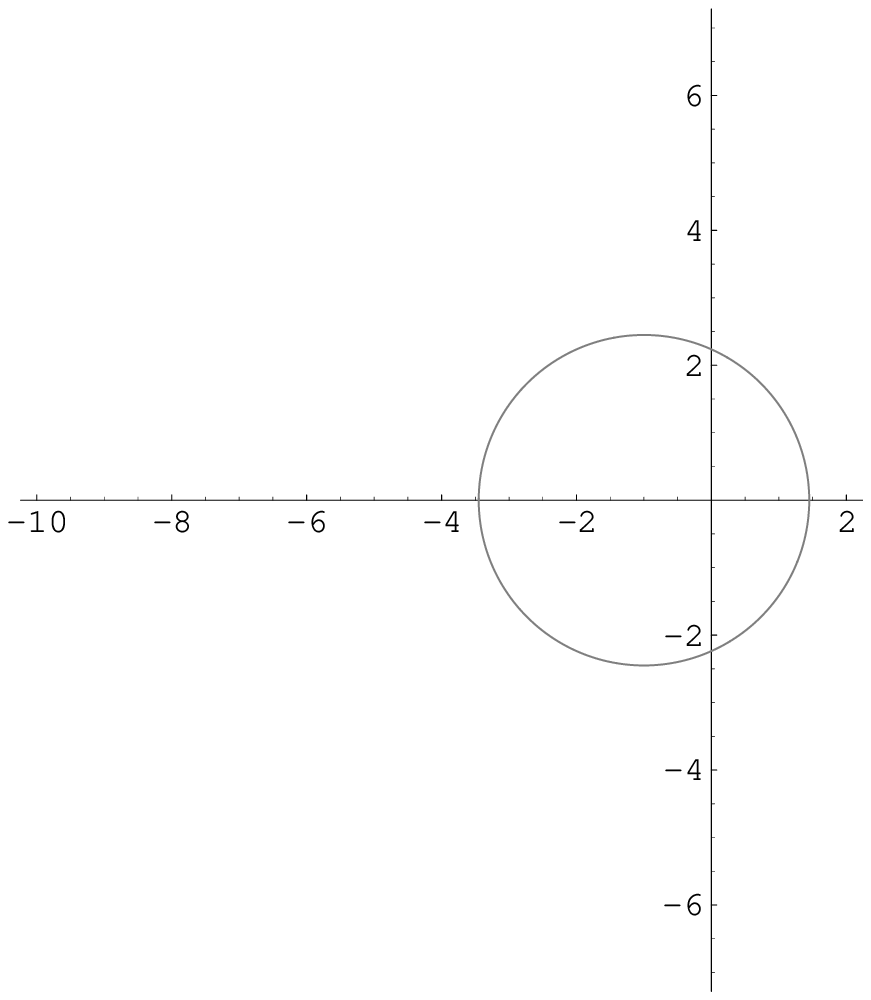}

  \includegraphics[width=.43\linewidth]{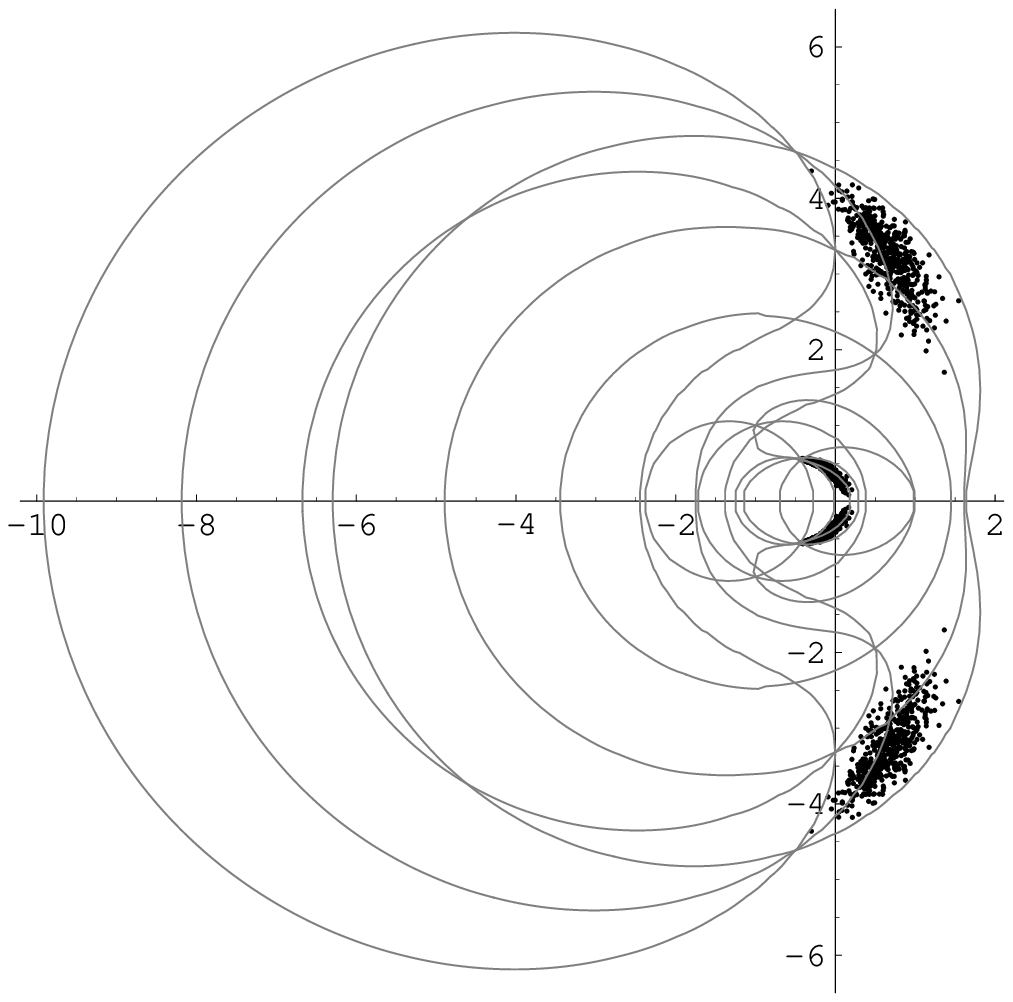}
  \caption{From left to right and top to bottom, the zero loci of
    $[\tW]_K$ for $d=4$ and $\e=\frac12$. Here $K$ runs through
    $\binom{\{0,\dots,d\}}{3}$ 
    in lexicographic order,
    except that the zero loci coresponding to $K=\{0,1,2\}$ and
    $K=\{1,2,3\}$ are empty and not shown. The last figure combines
    all the zero loci with the roots of $500$~random polynomials whose
    coefficients satisfy $\sum_{i=0}^{d}a_i - \frac{1}{\e}a_d=0$.}
  \label{fig:ChrEq}
\end{figure}

\section{Distribution of random roots}\label{sec:random}

In closing, we explain a phenomenon encountered several times in
the literature~\cite{Beck-etal05}, \cite{Braun-Develin06}: The roots
of ``randomly'' generated polynomials with nonnegative coefficients
tend to cluster together in several clumps, and usually lie well
inside the region permitted by theory; cf.~Figure~\ref{fig:cluster}.

\begin{figure}[htbp]
  \centering
  \includegraphics[width=.45\linewidth]{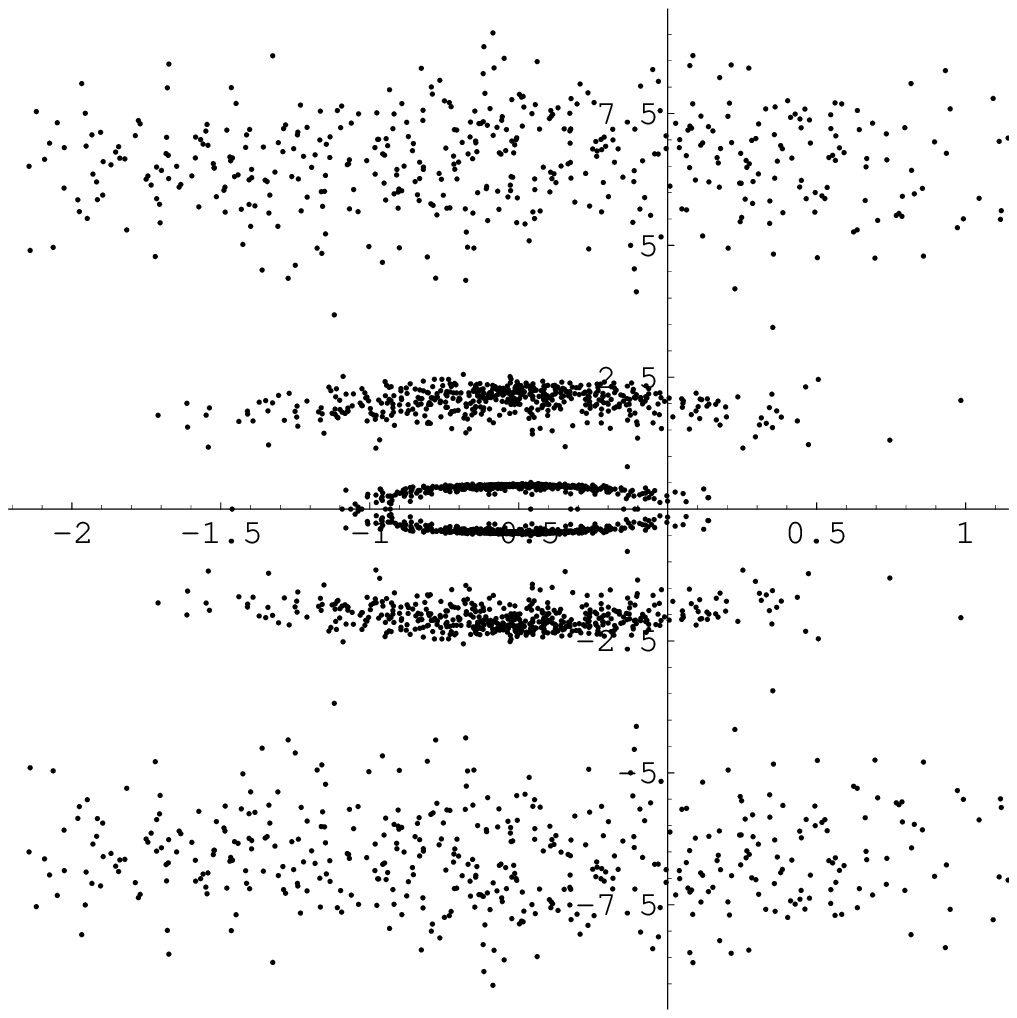} 
  \caption{The roots of 1000 random polynomials of
    degree $d=6$ with nonnegative coefficients in the binomial
    coefficient basis and $a_0,a_d\ne0$.}
  \label{fig:cluster}
\end{figure}
 
Our explanation is this: in these simulations, the coefficient vector
$(a_0,\dots,a_d)$ is usually picked uniformly at random from some cube
$[0,N]^{d+1}$ (except that sometimes the cases $a_0=0$ and $a_d=0$ are
excluded; we will gloss over this minor point). By linearity of
expectation, the expected value $E\big(f(z_0)\big)$ of
$f(z_0)=\sum_{i=0}^d a_i b_i(z_0)$ at a point $z_0\in\CC$ is
$\sum_{i=0}^d E(a_i) b_i(z_0) = \frac{N}{2}\sum_{i=0}^d b_i(z_0)$.
Thus, as a first approximation, the closer the barycenter
$\beta(z_0)=\sum_{i=0}^d b_i(z_0)$ is to zero, i.e., the smaller its
absolute value $|\beta(z_0)|$, the more likely it is for $z_0$ to be
a root of~$f$! For example, in the case of the binomial coefficient basis,
\[
   \beta(z_0) \ = \
   \sum_{i=0}^d \binom{z_0+d-i}{d} \ = \
   \sum_{i=0}^d \binom{z_0+i}{d} \ = \
   \binom{z_0+d+1}{d+1}-\binom{z_0}{d+1},
\]
by an elementary identity for binomial coefficients.
Figure~\ref{fig:contour} shows the regions where $|\beta(z_0)|$ is
small, together with the roots of several random polynomials. Note
that $\beta(z_0)$ is the Ehrhart polynomial of the simplex
${\rm conv}\{e_1,\dots,e_d,-e_1-\dots-e_d\}$
by~\cite[Proposition~1.3]{Henketal07}; see also~\cite{Villegas02}.

\begin{figure}[htbp]
  \centering
   \includegraphics[width=.4\linewidth]{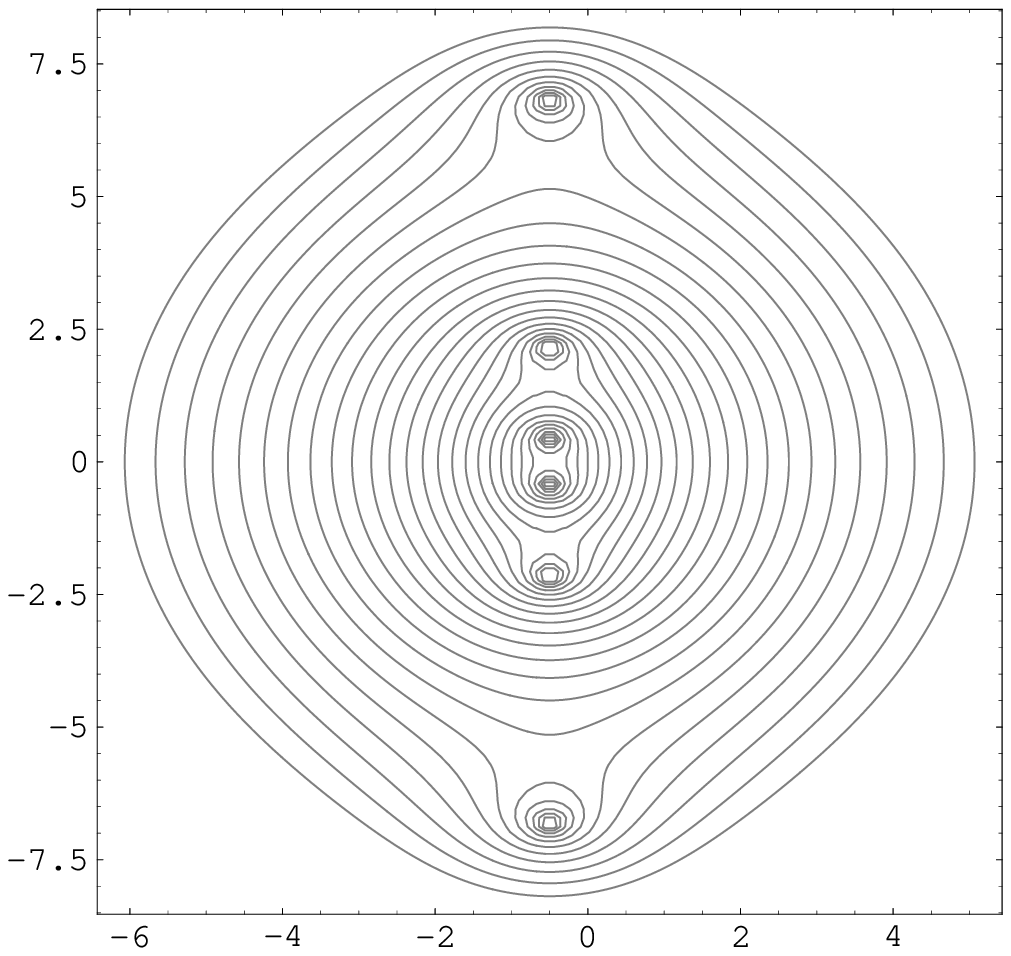}\qquad
   \includegraphics[width=.4\linewidth]{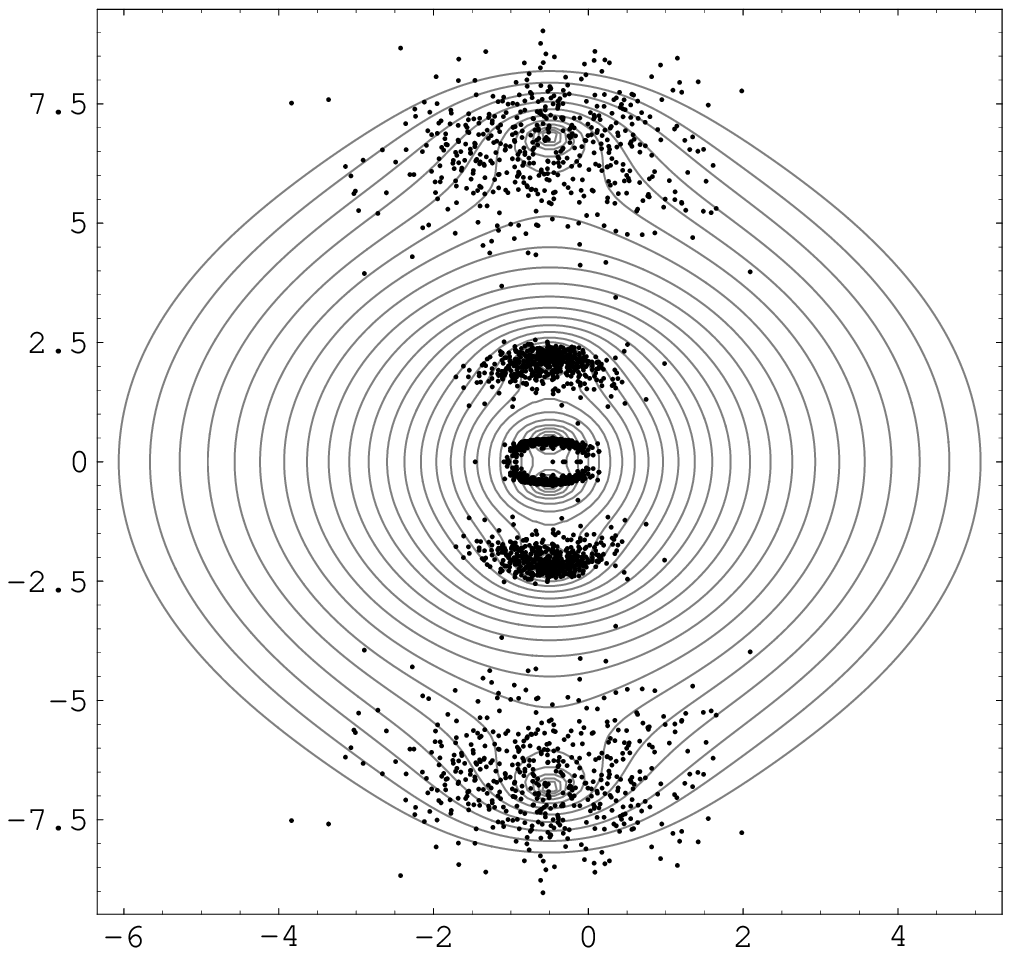}
   \caption{\emph{Left:} For $d=6$, the contours $|\beta(z_0)|=c$ for
     $b_i=\binom{z+d-i}{d}$ and varying~$c$; the innermost
     contours correspond to the smallest~$c$. \emph{Right:}
     additionally,  the roots of 500~polynomials of degree~$d$
     whose coefficients with respect to the $b_i$ are chosen uniformly
     at random from $[0,d!]$, except that $a_0,a_d\ne0$. }
  \label{fig:contour}
\end{figure}

Figure~\ref{fig:contourother} shows the corresponding regions for the
rising and falling factorial bases; in the case of the power basis
($b_i=z^i$), of course $\beta(z_0)=0$ iff $z_0\ne1$ is a $d$-th root
of unity.

\begin{figure}[htbp]
  \centering
  \includegraphics[width=.4\linewidth]{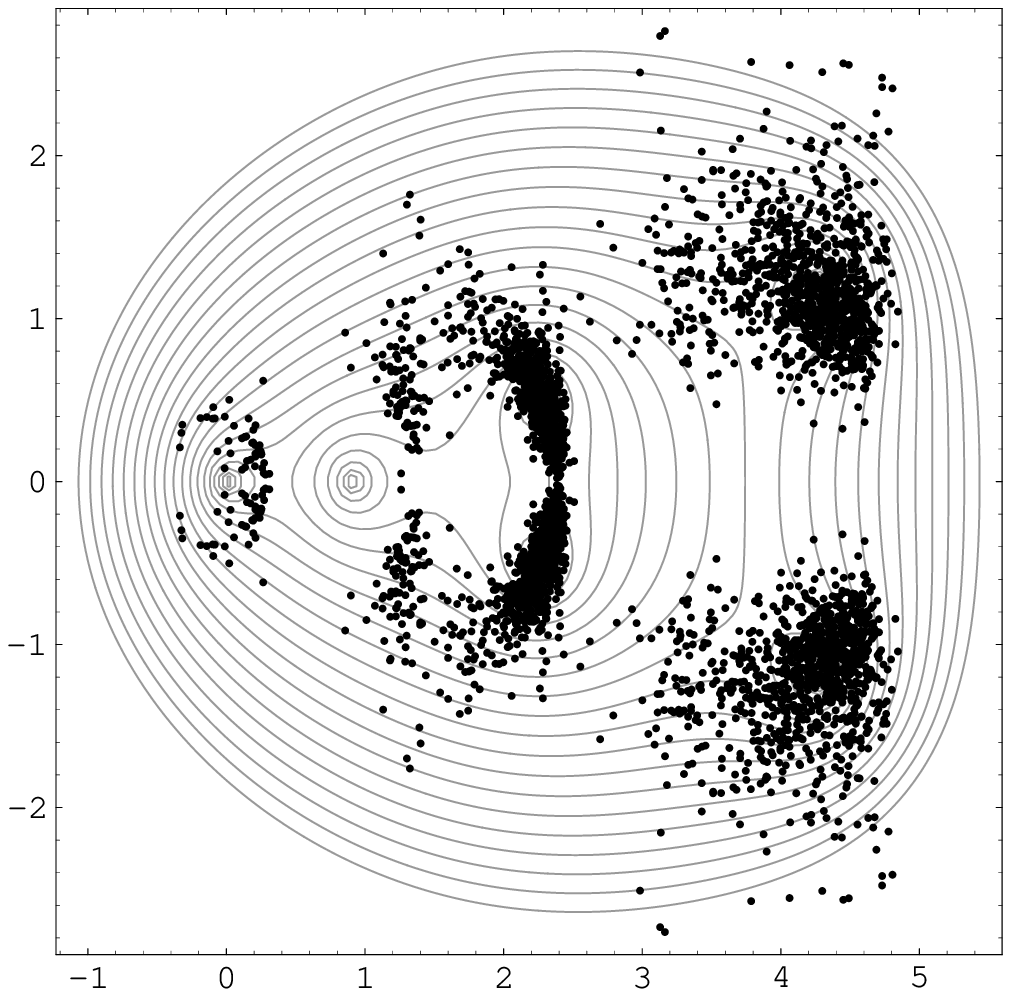}\qquad
  \includegraphics[width=.4\linewidth]{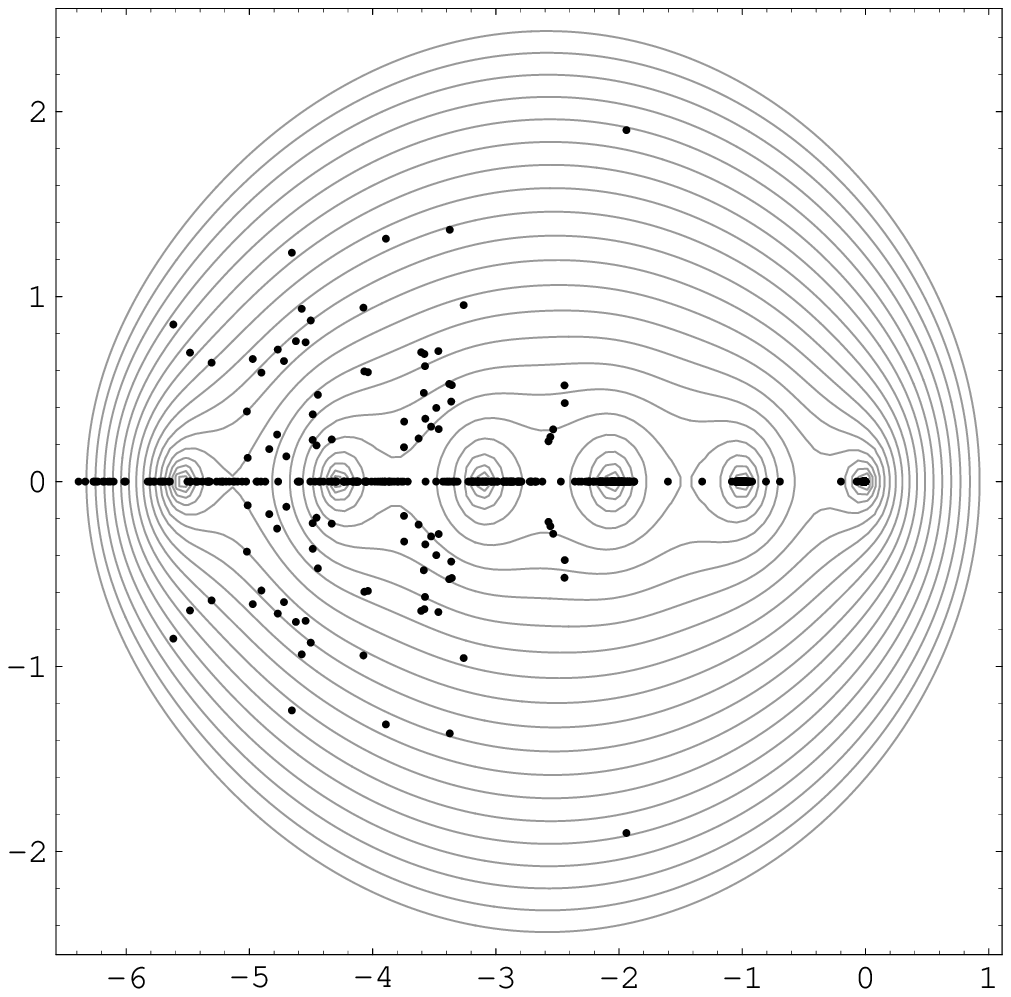}
  \caption{Contours $|\beta(z_0)|=c$ for $d=6$ and $b_i=z^{\underline
      i}$ (left), $b_i=z^{\overline{i}}$ (right), together with the
    roots of 1000, respectively 100 random polynomials with
    nonnegative coefficients with respect to these bases. For the rising
    factorial basis, the minima of $|\beta(z_0)|$ turn out to be real,
  so they only govern the distribution of the real zeros.}
  \label{fig:contourother}
\end{figure}

Clearly, the predictive power of this simple model can be easily
improved by considering additional parameters of the data; however, we
will not do this here.

\section*{Acknowledgements}

The author would like to thank Matthias Beck, Christian Haase and
G\"unter M.~Ziegler for helpful comments, and the two anonymous
referees for their careful reading and their pertinent suggestions
that have helped to improve the paper.

\section*{Appendix}

\begin{proof}[Proof of Lemma \ref{lem:leading}] It suffices to do the
  calculation for $D_n$ from equation~\eqref{eq:symDjk}.  So let's
  expand the difference
\begin{equation}\label{eq:diff}
   \prod_{i=1}^{n+1} (z-a_i)(\bz+a_i) - 
   \prod_{i=1}^{n+1} (\bz-a_i)(z+a_i),
\end{equation}
and pick out a term in the expansion with $l$ `$z$'s and $m-l$
`$\bz$'s. The coefficient of this term is a sum of terms of the form
\[
   (-1)^{m-l} a_{i_1}\cdots a_{i_{m-l}}a_{j_1}\cdots a_{j_l} -
   a_{i_1}\cdots a_{i_{m-l}} (-1)^l a_{i_{m-l}}a_{j_1}\cdots a_{j_l}, 
\]
and each of these terms vanishes for $m$ even. In particular, the term
$z^{n+1}\bz^{n+1}$ does not occur, which is also easy to see directly.
The first nonzero term in \eqref{eq:diff} is then
\[
   2 z^n\bz^{n+1} (-a_1-\dots-a_{n+1}) +
   2 z^{n+1}\bz^n (a_1+\dots+a_{n+1}) \ = \
   2z^n\bz^n(z-\bz)(a_1+\dots+a_{n+1}).
\]
It is  easy to work out $\sum_{i=1}^{n+1} a_i = d(n+1)/2$ for
$a_i=i+\Delta/2$, and this finishes the proof.
\end{proof}

\begin{proof}[Proof of Proposition \ref{prop:smooth}]
  The curve $\DD=\DD_n$ is also described by the equation
  $g=(h_1-h_2)/(\bz-z)$, where
  $h_1=\prod_{i=1}^{n+1}(z-a_i)(\bz+a_i)$,
  $h_2=\prod_{i=1}^{n+1}(z+a_i)(\bz-a_i)$, and $a_i=i+\Delta/2$ with
  $\Delta=d-2-n$. 
  Thus, $\DD$~has a singular point if
  and only if the Jacobi matrix of $g$ vanishes at some point of the
  locus $g=0$. Using the chain rule and the relations $\partial
  z/\partial x=1$, $\partial z/\partial y=i$, $\partial\bz/\partial x
  = 1$, $\partial \bz/\partial y = -i$, we calculate the partial
  derivatives of $g(z)$ with respect to $x$ and $y$:
  \begin{eqnarray*}
    \frac{\partial g(z)}{\partial x} &=&
    \frac{h_{1,z}-h_{2,z}+h_{1,\bz}-h_{2,\bz}}{\bz-z},\\
    \frac{\partial g(z)}{\partial y} &=&
    2i\,\frac{h_1-h_2}{(\bz-z)^2} + i\,
    \frac{h_{1,z}-h_{2,z}-h_{1,\bz}+h_{2,\bz}}{\bz-z}.
  \end{eqnarray*}
  Here $h_{j,z}$, $h_{j,\bz}$ denote the partial derivatives of
  $h_j$ with respect to $z$, $\bz$; by explicit differentiation,
  $h_{1,z}=\sum_{i=1}^{n+1} h_1/(z-a_i)$ and
  $h_{2,z}=\sum_{i=1}^{n+1} h_2/(z+a_i)$. 

  To prove that $\DD$ has no real singular points, we pick $z\in\DD$
  and calculate
  \begin{eqnarray*}
     \frac{\partial g}{\partial x}(z) &=&
     \frac{h_{1,z}+h_{1,\bz}}{\bz-z} - 
     \frac{h_{2,z}+h_{2,\bz}}{\bz-z} \notag\\
     &=&
     \frac{h_1}{\bz-z}
     \sum_{i=1}^{n+1} \left(
       \frac{1}{z-a_i} + \frac{1}{\bz+a_i}
       - \frac{1}{z+a_i} - \frac{1}{\bz-a_i}
     \right)\notag\\
     &=&
     \frac{h_1}{\bz-z}
     \sum_{i=1}^{n+1}
     \frac{2a_i(\bz^2-z^2)}{(z^2-a_i^2)(\bz^2-a_i^2)} \notag\\
     &=&
     4\,h_1
     \sum_{i=1}^{n+1}
     \frac{x  a_i}{\big((x-a_i)^2+y^2\big)\big((x+a_i)^2+y^2\big)}.
  \end{eqnarray*}
  For real nonzero $z\in\DD$, this expression never vanishes.  The
  same calculation already proves the second statement, because a
  tangent vector to the curve $g(x,y)=0$ at a non-singular point
  $(x_0,y_0)$ is given by $\pm\big({-\frac{\partial g}{\partial
      y}}(x_0,y_0),\ \frac{\partial g}{\partial x}(x_0,y_0)\big)$.
  
  We now examine a non-real singular point $z_0$ of $\DD$. Any such
  point must satisfy
  \[
     \frac{h_1(z_0)-h_2(z_0)}{\bz_0-z_0} \ = \ 
     0 \ = \ h_{1,z}(z_0)-h_{2,z}(z_0).
  \]
  The first equation tells us that $h_1(z_0)=h_2(z_0)$, so that
  $h_{1,z}(z_0)-h_{2,z}(z_0)=0$ if and only if $h_1(z_0)=0$
  (which is incompatible with $z_0\notin\RR$ and $g(z_0)=0$), or
  \[
     0 \ = \ \sum_{i=1}^{n+1}\frac{1}{z_0-a_i} - 
     \sum_{i=1}^{n+1}\frac{1}{z_0+a_i}\ = \
     2 \sum_{i=1}^{n+1}\frac{a_i}{z_0^2-a_i^2}.
  \]
  Writing $z_0^2=x_0+iy_0$ and separating the real and imaginary parts in
  the last expression yields
  \[
    \sum_{i=1}^{n+1}\frac{a_i}{(x_0-a_i)^2+y_0^2} \ = \
    \sum_{i=1}^{n+1}\frac{a_i^3}{(x_0-a_i)^2+y_0^2} \ = \ 0.
  \]
  But the denominators of these expressions are positive (the $a_i$
  and the origin do not lie on~$\DD$), and $a_i>0$ for
  $i=1,\dots,n+1$, so we conclude that $\DD$ has no singular points.
\end{proof}

\bibliographystyle{siam}
\bibliography{ehrhartzeros}

\begin{thebibliography}{10}

\bibitem{Beck-etal05}
{\sc M.~Beck, J.~A. De~Loera, M.~Develin, J.~Pfeifle, and R.~P. Stanley}, {\em
  Coefficients and roots of {E}hrhart polynomials}, in Integer points in
  polyhedra---geometry, number theory, algebra, optimization, vol.~374 of
  Contemp. Math., Amer. Math. Soc., Providence, RI, 2005, pp.~15--36.

\bibitem{Henketal07}
{\sc C.~Bey, M.~Henk, and J.~M. Wills}, {\em {Notes on the roots of Ehrhart
  polynomials.}}, Discrete Comput. Geom., 38 (2007), pp.~81--98.

\bibitem{Braun06}
{\sc B.~Braun}, {\em Norm bounds for {E}hrhart polynomial roots}.
\newblock \texttt{arXiv:math.CO/0602464}, 3~pages.

\bibitem{Braun-Develin06}
{\sc B.~Braun and M.~Develin}, {\em Ehrhart polynomial roots and {S}tanley's
  non-negativity theorem}.
\newblock \texttt{arXiv:math.CO/0610399}, 11~pages.

\bibitem{Brenti92}
{\sc F.~Brenti}, {\em Expansions of chromatic polynomials and log-concavity},
  Trans. Amer. Math. Soc., 332 (1992), pp.~729--756.

\bibitem{Brenti-Royle-Wagner94}
{\sc F.~Brenti, G.~F. Royle, and D.~G. Wagner}, {\em {Location of zeros of
  chromatic and related polynomials of graphs.}}, Can. J. Math., 46 (1994),
  pp.~55--80.

\bibitem{Hall-Siry-Vanderslice1965}
{\sc D.~Hall, J.~Siry, and B.~Vanderslice}, {\em {The chromatic polynomial of
  the truncated icosahedron.}}, Proc. Am. Math. Soc., 16 (1965), pp.~620--628.

\bibitem{Read-Tutte88}
{\sc R.~Read and W.~Tutte}, {\em {Chromatic polynomials.}}
\newblock {Selected topics in graph theory, Vol. 3, 15-42 (1988).}, 1988.

\bibitem{Villegas02}
{\sc F.~Rodriguez-Villegas}, {\em {On the zeros of certain polynomials.}},
  Proc. Am. Math. Soc., 130 (2002), pp.~2251--2254.

\bibitem{Sokal2004}
{\sc A.~D. Sokal}, {\em {Chromatic roots are dense in the whole complex
  plane.}}, Comb. Probab. Comput., 13 (2004), pp.~221--261.

\bibitem{Stanley80}
{\sc R.~P. Stanley}, {\em Decompositions of rational convex polytopes}, Ann.
  Discrete Math., 6 (1980), pp.~333--342.
\newblock Combinatorial mathematics, optimal designs and their applications
  (Proc. Sympos. Combin. Math. and Optimal Design, Colorado State Univ., Fort
  Collins, Colo., 1978).

\bibitem{Ziegler98}
{\sc G.~M. Ziegler}, {\em {Lectures on polytopes.}}, {Graduate Texts in
  Mathematics. 152. Springer-Verlag}, 1995.

\end{thebibliography}

\end{document}